\def\DateTime{29/April/2010, 11:30 (Kyoto)}
\def\Version{Version $3.0$}
\def\yes{\if00}
\def\no{\if01}
\def\iftenpt{\no}
\def\ifelevenpt{\yes}
\def\iftwelvept{\no}

\def\ifusepdf{\no}
\def\ifpsfont{\yes}
\def\ifquery{\yes}

\iftenpt
\documentclass[leqno]{amsart}
\else\fi
\ifelevenpt
\documentclass[leqno,11pt]{amsart}
\else\fi
\iftwelvept
\documentclass[leqno,12pt]{amsart}
\else\fi

\usepackage{amssymb}
\usepackage{amscd}
\usepackage{latexsym}
\usepackage{verbatim}
\usepackage{mathrsfs}
\usepackage{eepic}
\usepackage[all]{xy}

\ifusepdf
\usepackage{hyperref}
\else\fi
\ifpsfont
\usepackage[T1]{fontenc}
\usepackage{times}
\else\fi

\ifelevenpt
\else\fi

\iftwelvept
\else\fi

\theoremstyle{plain}
\newtheorem{Theorem}{Theorem}[section]
\newtheorem{Proposition}[Theorem]{Proposition}
\newtheorem{Lemma}[Theorem]{Lemma}

\newtheorem{Corollary}[Theorem]{Corollary}

\newtheorem{Claim}{Claim}[Theorem]
\newtheorem{Main}{Main Results}

\theoremstyle{definition}

\newtheorem{Remark}[Theorem]{Remark}

\newtheorem{Convention}{}[section]

\def\rom{\textup}
\newcommand{\ZZ}{{\mathbb{Z}}}
\newcommand{\QQ}{{\mathbb{Q}}}
\newcommand{\RR}{{\mathbb{R}}}

\newcommand{\CC}{{\mathbb{C}}}
\newcommand{\PP}{{\mathbb{P}}}

\newcommand{\OO}{{\mathcal{O}}}

\newcommand{\Proj}{\operatorname{Proj}}

\newcommand{\Rat}{\operatorname{Rat}}

\newcommand{\Sing}{\operatorname{Sing}}

\newcommand{\Supp}{\operatorname{Supp}}

\newcommand{\mult}{\operatorname{mult}}

\newcommand{\adeg}{\widehat{\operatorname{deg}}}

\newcommand{\Bs}{\operatorname{Bs}}

\newcommand{\dist}{\operatorname{dist}}
\newcommand{\Div}{\operatorname{Div}}

\newcommand{\avol}{\widehat{\operatorname{vol}}}
\newcommand{\aH}{\hat{H}^0}
\newcommand{\ah}{\hat{h}^0}

\newcommand{\Tpsh}{\operatorname{PSH}}

\newcommand{\Conv}{\operatorname{Conv}}

\newcommand{\rest}[2]{\left.{#1}\right\vert_{{#2}}}
\ifquery
\def\query#1{\setlength\marginparwidth{80pt}%
\marginpar{\raggedright\fontsize{7.81}{9} 
\selectfont\upshape\hrule\smallskip 
#1\par\smallskip\hrule}} 

\else
\def\query#1{}
\fi

\begin{document}

\title[Big arithmetic divisors on the projective spaces over $\ZZ$]%
{Big arithmetic divisors on the projective spaces over $\ZZ$}
\author{Atsushi Moriwaki}
\address{Department of Mathematics, Faculty of Science,
Kyoto University, Kyoto, 606-8502, Japan}
\email{moriwaki@math.kyoto-u.ac.jp}
\date{\DateTime, (\Version)}


\maketitle

\setcounter{tocdepth}{1}
\tableofcontents

\renewcommand{\theTheorem}{\Alph{Theorem}}
\section*{Introduction}
Let $\PP^n_{\ZZ} = \Proj(\ZZ[T_0, T_1, \ldots, T_n])$, $H_i = \{ T_i = 0 \}$ and $z_i = T_i/T_0$ for $i=0, 1, \ldots, n$.
Let us fix a sequence $\pmb{a} = (a_0, a_1, \ldots, a_n)$ of positive numbers. We define a $H_0$-Green function $g_{\pmb{a}}$ of ($C^{\infty} \cap \Tpsh$)-type 
on $\PP^n(\CC)$ and
an arithmetic divisor $\overline{D}_{\pmb{a}}$ of $(C^{\infty} \cap \Tpsh)$-type on $\PP^n_{\ZZ}$
to be
\[
g_{\pmb{a}} := \log (a_0 + a_1 \vert z_1 \vert^2 + \cdots + a_n \vert z_n \vert^2)\quad\text{and}\quad
\overline{D}_{\pmb{a}} := (H_0, g_{\pmb{a}}).
\]
In this paper, we will observe several properties of $\overline{D}_{\pmb{a}}$ and 
give the exact form of the Zariski decomposition of $\overline{D}_{\pmb{a}}$ on $\PP^1_{\ZZ}$.
Further, we will show that, if $n\geq 2$ and $\overline{D}_{\pmb{a}}$ is big and not nef,
then, for any birational morphism
$f: X \to \PP^n_{\ZZ}$ of projective, generically smooth and normal arithmetic varieties,
we can not expect a suitable Zariski decomposition of $f^*(\overline{D}_{\pmb{a}})$.
In this sense, the results in \cite{MoArZariski} are nothing short of miraculous, and
arithmetic linear series are very complicated and have richer structure than what we expected.
We also give a concrete construction of Fujita's approximation of $\overline{D}_{\pmb{a}}$.
The following is a list of the main results of this paper.

\renewcommand{\theMain}{\!}
\begin{Main}
Let $\varphi_{\pmb{a}} : \RR_{\geq 0}^{n+1} \to \RR$ be a function given by
\[
\varphi_{\pmb{a}}(x_0, x_1, \ldots, x_n) := - \sum_{i=0}^n x_i \log x_i + \sum_{i=0}^n x_i \log a_i,
\]
and let 
\[
\Theta_{\pmb{a}} := \left\{ (x_1, \ldots, x_n) \in \Delta_n  \mid 
\varphi_{\pmb{a}}(1-x_1-\cdots-x_n,
x_1, \ldots, x_n) \geq 0 \right\},
\]
where $\Delta_n := \left\{ (x_1, \ldots, x_n) \in \RR_{\geq 0}^n \mid x_1 + \cdots + x_n \leq 1 \right\}$.
Then the following properties hold for $\overline{D}_{\pmb{a}}$:

\begin{enumerate}
\renewcommand{\labelenumi}{(\arabic{enumi})}
\item
$\overline{D}_{\pmb{a}}$ is ample if and only if $a_0 > 1, a_1 > 1, \ldots, a_n > 1$.

\item
$\overline{D}_{\pmb{a}}$ is nef if and only if $a_0 \geq 1, a_1 \geq 1, \ldots, a_n \geq 1$.

\item
$\overline{D}_{\pmb{a}}$ is big if and only if $a_0 + a_1 + \cdots + a_n > 1$.

\item
$\overline{D}_{\pmb{a}}$ is pseudo-effective if and only if $a_0 + a_1 + \cdots + a_n \geq 1$.

{\nopagebreak
\begin{figure}[h]
\begin{center}
\unitlength=0.5mm
\begin{picture}(110,110)
\put(-5,0){\vector(1,0){100}}
\put(0,-5){\vector(0,1){100}}
\put(0,50){\line(1,-1){50}}
\put(50,50){\line(1,0){50}}
\put(50,50){\line(0,1){50}}
\put(97,-1){$a_0$}
\put(-1,97){$a_1$}
\put(13,15){\text{\tiny \rom{Not}}}
\put(2,10){\text{\tiny \rom{Pseudo-effective}}}
\put(70,70){\text{\tiny \rom{Ample}}}
\put(35,35){\text{\tiny \rom{Big}}}
\put(80,40){\vector(0,1){10}}
\put(60,35){\text{\tiny \rom{Nef on the boundary}}}
\put(40,20){\vector(-1,0){10}}
\put(41,18.5){\text{\tiny \rom{Pseudo-effective on the boundary}}}
\put(50,0){\circle*{2}}
\put(0,50){\circle*{2}}
\put(50,50){\circle*{2}}
\put(45,-5){\tiny \rom{(1,0)}}
\put(-11,47){\tiny \rom{(0,1)}}
\put(45,45){\tiny \rom{(1,1)}}
\end{picture}
\caption{Geography of $\overline{D}_{\pmb{a}}$ on $\PP^1_{\ZZ}$}
\end{center}
\end{figure}}

\item
$\aH(\PP^n_{\ZZ}, l \overline{D}_{\pmb{a}}) \not= \{ 0 \}$ if and only if
$l \Theta_{\pmb{a}} \cap \ZZ^n \not= \emptyset$. As consequences, we have the following:

\begin{enumerate}
\renewcommand{\labelenumii}{(\arabic{enumi}.\arabic{enumii})}
\item
We assume that $a_0 + a_1 + \cdots + a_n = 1$. For a positive integer $l$,
\[
\aH(\PP^n_{\ZZ}, l \overline{D}_{\pmb{a}}) = \begin{cases}
\{ 0, \pm z_1^{la_1} \cdots z_n^{la_n} \} & \text{if $la_1, \ldots, la_n \in \ZZ$},\\
\{ 0 \} & \text{otherwise}.
\end{cases}
\]
In particular, if $\pmb{a} \not\in \QQ^{n+1}$, then $\aH(\PP^n_{\ZZ}, l \overline{D}_{\pmb{a}}) = \{ 0 \}$ for all $l \geq 1$.

\item
For any positive integer $l$,
there exists $\pmb{a} \in \QQ_{>0}^{n+1}$ such that
$\overline{D}_{\pmb{a}}$ is big and 
\[
\aH(\PP^n_{\ZZ}, k \overline{D}_{\pmb{a}} ) = \{ 0 \}
\]
for all $k$ with
$1 \leq k \leq l$.
\end{enumerate}

\item
${\displaystyle \left\langle \aH(\PP^n_{\ZZ}, l \overline{D}_{\pmb{a}}) \right\rangle_{\ZZ} = \bigoplus_{(e_1, \ldots, e_n) \in l \Theta_{\pmb{a}} \cap \ZZ^{n}} \ZZ z_1^{e_1} \cdots z_n^{e_n}}$
if $l \Theta_{\pmb{a}} \cap \ZZ^n \not= \emptyset$.

\item
\rom{(}Integral formula\rom{)}\quad The following formulae hold:
\begin{align*}
\hspace{4em}&\avol(\overline{D}_{\pmb{a}}) = \frac{(n+1)!}{2} \int_{\Theta_{\pmb{a}}} \varphi_{\pmb{a}}(1-x_1 - \cdots - x_n, x_1, \ldots, x_n) 
dx_1 \cdots dx_n, \\
\intertext{and}
& \adeg(\overline{D}_{\pmb{a}}^{n+1}) = \frac{(n+1)!}{2} \int_{\Delta_n} \varphi_{\pmb{a}}(1-x_1 - \cdots - x_n, x_1, \ldots, x_n) 
dx_1 \cdots dx_n.
\end{align*}
In particular, $\adeg(\overline{D}_{\pmb{a}}^{n+1}) = \avol(\overline{D}_{\pmb{a}})$ if and only if
$\overline{D}_{\pmb{a}}$ is nef.

\item
\rom{(}Zariski decomposition for $n=1$\rom{)}\quad
We assume $n=1$.
The Zariski decomposition of $\overline{D}_{\pmb{a}}$ exists if and only if $a_0 + a_1 \geq 1$.
Moreover, the positive part of $\overline{D}_{\pmb{a}}$ is given by $(\theta_{\pmb{a}}H_0 - \vartheta_{\pmb{a}}H_{1}, p_{\pmb{a}})$,
where $\vartheta_{\pmb{a}} = \inf \Theta_{\pmb{a}}$, $\theta_{\pmb{a}} = \sup \Theta_{\pmb{a}}$ and
\[
\hspace{4em}
p_{\pmb{a}}(z_1) =\begin{cases}
\vartheta_{\pmb{a}} \log \vert z_1 \vert^2 & \text{if $\vert z_1 \vert < \sqrt{\frac{a_0\vartheta_{\pmb{a}}}{a_1(1-\vartheta_{\pmb{a}})}}$}, \\
\log (a_0   + a_1\vert z_1 \vert^2) & 
\text{if $\sqrt{\frac{a_0\vartheta_{\pmb{a}}}{a_1(1-\vartheta_{\pmb{a}})}} \leq \vert z_1 \vert \leq \sqrt{\frac{a_0\theta_{\pmb{a}}}{a_1(1-\theta_{\pmb{a}})}}$}, \\
\theta_{\pmb{a}} \log \vert z_1 \vert^2 & \text{if $\vert z_1 \vert > \sqrt{\frac{a_0\theta_{\pmb{a}}}{a_1(1-\theta_{\pmb{a}})}}$},
\end{cases}
\]
In particular, if $a_0 + a_1 = 1$, then the positive part is $-a_1\widehat{(z_1)}$.

\item
\rom{(}Impossibility of Zariski decomposition for $n \geq 2$\rom{)}\quad
We assume $n \geq 2$. If $\overline{D}_{\pmb{a}}$ is big and not nef 
\rom{(}i.e., $a_0 + \cdots + a_n > 1$ and
$a_i < 1$ for some $i$\rom{)},
then, for any birational morphism $f : X \to \PP^n_{\ZZ}$ of projective, generically smooth and normal arithmetic varieties,
there is no decomposition $f^*(\overline{D}_{\pmb{a}}) = \overline{P} + \overline{N}$ with the following properties:
\begin{enumerate}
\renewcommand{\labelenumii}{(\arabic{enumi}.\arabic{enumii})}
\item
$\overline{P}$ is a nef and big arithmetic $\RR$-divisor of $(C^0 \cap \Tpsh)$-type on $X$.

\item
$\overline{N}$ is an effective arithmetic $\RR$-divisor of $C^0$-type on $X$.

\item For any horizontal prime divisor $\Gamma$ on $X$ \rom{(}i.e. $\Gamma$ is a reduced and irreducible divisor on $X$ such that $\Gamma$ is flat over $\ZZ$\rom{)},
\begin{multline*}
\hspace{3em}
\mult_{\Gamma}(N) \\
\hspace{6em}
\leq \inf \left\{ \mult_{\Gamma}(f^*(H_0) + (1/l)(\phi)) \mid l \in \ZZ_{>0},\ \phi \in \aH(lf^*(\overline{D}_{\pmb{a}})) \setminus \{ 0 \} \right\}.
\end{multline*}
\end{enumerate}

\item \rom{(}Fujita's approximation\rom{)}
We assume that $\overline{D}_{\pmb{a}}$ is big.
Let $\operatorname{Int}(\Theta_{\pmb{a}})$ be the set of interior points of $\Theta_{\pmb{a}}$.
We choose $\pmb{x}_1, \ldots, \pmb{x}_r \in \operatorname{Int}(\Theta_{\pmb{a}}) \cap \QQ^n$ such that 
\[
\frac{(n+1)!}{2} \int_{\Theta} \phi_{(\pmb{x}_1, \varphi_{\pmb{a}}(\widetilde{\pmb{x}}_1)), \ldots, (\pmb{x}_r, \varphi_{\pmb{a}}(\widetilde{\pmb{x}}_r))}(\pmb{x}) d \pmb{x} > \avol(\overline{D}_{\pmb{a}}) - \epsilon,
\]
where
$\Theta := \Conv \{ \pmb{x}_1, \ldots, \pmb{x}_r \}$ and
\begin{multline*}
\hspace{3em}
\phi_{(\pmb{x}_1, \varphi_{\pmb{a}}(\widetilde{\pmb{x}}_1)), \ldots, (\pmb{x}_r, \varphi_{\pmb{a}}(\widetilde{\pmb{x}}_r))}(\pmb{x}) := \\
\qquad
\max \{ t \in \RR \mid (\pmb{x}, t) \in \Conv \{ (\pmb{x}_1, \varphi_{\pmb{a}}(\widetilde{\pmb{x}}_1)),
\ldots, (\pmb{x}_r, \varphi_{\pmb{a}}(\widetilde{\pmb{x}}_r)) \} \subseteq \RR^n \times \RR \}
\end{multline*}
for $\pmb{x} \in \Theta$ \rom{(}see
Conventions and terminology~\rom{\ref{CT:lifting}} for the definition of $\widetilde{\pmb{x}}_1, \ldots, \widetilde{\pmb{x}}_r$\rom{)}.
Using the above points $\pmb{x}_1, \ldots, \pmb{x}_r$,
we can construct a birational morphisms $\mu : Y \to \PP^n_{\ZZ}$ of projective, generically smooth and normal arithmetic varieties,
and a nef arithmetic $\QQ$-divisor $\overline{P}$ of $(C^{\infty} \cap \Tpsh)$-type on $Y$ such that
\[
\overline{P} \leq \mu^*(\overline{D}_{\pmb{a}})\quad\text{and}\quad
\avol(\overline{P}) > \avol(\overline{D}_{\pmb{a}}) - \epsilon.
\]
For details, see Section~\rom{\ref{sec:Fujita:app}}.
\end{enumerate}
\end{Main}

I would like to express my thanks to Prof. Yuan.
The studies of this paper  started from his question.
I thank  Dr. Uchida.
Without his calculation of the limit of a sequence,
I could not find the positive part of $\overline{D}_{\pmb{a}}$ on $\PP^1_{\ZZ}$.
In addition, I also thank Dr. Hajli for his comments. 

\bigskip
\renewcommand{\theConvention}{\arabic{Convention}}
\renewcommand{\theequation}{CT.\arabic{Convention}.\arabic{Claim}}
\subsection*{Conventions and terminology}

\begin{Convention}
\label{CT:entry}
For $\pmb{x} =(x_1, \ldots, x_r) \in \RR^r$,
the $i$-th entry $x_i$ of $\pmb{x}$ is denoted by $\pmb{x}(i)$. We define 
$\vert \pmb{x} \vert$ to be $\vert \pmb{x} \vert := x_1 + \cdots + x_r$.
\end{Convention}

\begin{Convention}
\label{CT:lifting}
For $\pmb{x} = (x_1, \ldots, x_r) \in \RR^r$ and $m \in \RR$, we define $\widetilde{\pmb{x}}^m \in \RR^{r+1}$ to be
\[
\widetilde{\pmb{x}}^m = (m - x_1 - \cdots - x_r, x_1, \ldots, x_r).
\]
Note that $\vert \widetilde{\pmb{x}}^m \vert = m$.
For simplicity, in the case where $m=1$,
we denote $\widetilde{\pmb{x}}^m$ by $\widetilde{\pmb{x}}$.
\end{Convention}

\begin{Convention}
\label{CT:mult:coeff}
Let $\pmb{e} = (e_1, \ldots, e_r) \in \ZZ_{\geq 0}^r$ and $l = \vert \pmb{e} \vert$.
A monomial $z_1^{e_1} \cdots z_r^{e_r}$ is denoted by $z^{\pmb{e}}$.
The multinomial coefficient
${\displaystyle \frac{l!}{e_1! \cdots e_r!}}$
is denoted by ${\displaystyle \binom{l}{\pmb{e}}}$.
\end{Convention}

\renewcommand{\theTheorem}{\arabic{section}.\arabic{Theorem}}
\renewcommand{\theClaim}{\arabic{section}.\arabic{Theorem}.\arabic{Claim}}
\renewcommand{\theequation}{\arabic{section}.\arabic{Theorem}.\arabic{Claim}}

\section{Fundamental properties of the characteristic function}
\label{sec:fund:prop:func}
Let $\PP^n_{\ZZ} = \Proj(\ZZ[T_0, T_1, \ldots, T_n])$, $H_i = \{ T_i = 0 \}$ and $z_i = T_i/T_0$
for $i=0, \ldots, n$.
Let us fix $\pmb{a} = (a_0, a_1, \ldots, a_n) \in \RR_{>0}^{n+1}$.
We set 
\[
h_{\pmb{a}} = a_0  + a_1 \vert z_1 \vert^2 + \cdots + a_n \vert z_n \vert^2,\quad
g_{\pmb{a}} = \log h_{\pmb{a}} \quad\text{and}\quad
\omega_{\pmb{a}} = dd^c(g_{\pmb{a}})
\]
on $\PP^n(\CC)$, that is,
\[
g_{\pmb{a}} = -\log\vert T_0 \vert^2 + \log\left(a_0 \vert T_0 \vert^2 + \cdots + a_n \vert T_n \vert^2\right).
\]

\begin{Proposition}
\label{prop:calculation:volume:form}
\begin{enumerate}
\renewcommand{\labelenumi}{(\arabic{enumi})}
\item
$\omega_{\pmb{a}}$ is positive.
In particular, $g_{\pmb{a}}$ is a $H_0$-Green function of $(C^{\infty} \cap \Tpsh)$-type.

\item
If we set $\Phi_{\pmb{a}} = \omega_{\pmb{a}}^{\wedge n}$,
then
\[
\Phi_{\pmb{a}} = \left(\frac{\sqrt{-1}}{2\pi}\right)^n \frac{n! a_0 \cdots a_n}
{h_{\pmb{a}}^{n+1}}
dz_1 \wedge d\bar{z}_1 \wedge \cdots \wedge dz_n \wedge d\bar{z}_{n}
\]
and ${\displaystyle \int_{\PP^n(\CC)} \Phi_{\pmb{a}} = 1}$.
\end{enumerate}
\end{Proposition}

\begin{proof}
(1) Note that
\[
\omega_{\pmb{a}} = \frac{\sqrt{-1}}{2\pi} \left( \sum_{i=1}^n \frac{a_i}{h_{\pmb{a}}(z)} dz_i \wedge d\bar{z}_i
- \sum_{i,j} \frac{a_ia_j\bar{z}_iz_j}{h_{\pmb{a}}(z)^2} dz_i \wedge d\bar{z}_j \right).
\]
We set
\[
A = \left( \delta_{ij}\frac{a_i}{h_{\pmb{a}}(z)} - \frac{a_ia_j\bar{z}_iz_j}{h_{\pmb{a}}(z)^2} \right)_{\substack{1 \leq i \leq n, \\
1 \leq j \leq n}}.
\]
Then it is easy to see that
\[
\begin{pmatrix} \bar{\lambda}_1 & \cdots & \bar{\lambda}_n \end{pmatrix} A \begin{pmatrix} \lambda_1 \\ \vdots \\ \lambda_n \end{pmatrix} =
\frac{a_0 \sum_{i=1}^n a_i \vert \lambda_i \vert^2 + \sum_{i < j} a_i a_j \vert z_i \bar{\lambda}_j - z_j \bar{\lambda}_i \vert^2}{h_{\pmb{a}}(z)^2}.
\]
Thus $A$ is positive definite.

\medskip
(2) The first assertion follows from the following claim:

\begin{Claim}
For $\alpha_1, \ldots, \alpha_n \in \CC$,
\[
\det \begin{pmatrix} \delta_{ij}
t_i - \alpha_i\bar{\alpha}_j
\end{pmatrix}_{\substack{1 \leq i \leq n \\ 1 \leq j \leq n}} = t_1 \cdots t_n - \sum_{i=1}^n \vert \alpha_i \vert^2 t_1 \cdots t_{i-1} \cdot
t_{i+1} \cdots t_n.
\]
\end{Claim}

\begin{proof}
We denote $\begin{pmatrix} \delta_{ij}t_i - \alpha_i\bar{\alpha}_j \end{pmatrix}_{\substack{1 \leq i \leq n \\ 1 \leq j \leq n}}$ by $B$.
If $t_i = t_j = 0$ for $i \not= j$, then
the $i$-the column and the $j$-the column of $B$ are linearly dependent, so that $\det B = 0$.
Therefore, we can set
\[
\det B = t_1 \cdots t_n -\sum_{i=1}^n c_i t_1 \cdots t_{i-1} \cdot t_{i+1} \cdots t_n
\]
for some $c_1, \ldots, c_n \in \CC$.
It is easy to see that $\det B = -\vert \alpha_i \vert^2$ if $t_i = 0$ and
$t_1 = \cdots = t_{i-1} = t_{i+1} = \cdots = t_n = 1$.
Thus $c_i = \vert \alpha_i \vert^2$.
\end{proof}

Let $\vert \cdot\vert_{\pmb{a}}$ be a $C^{\infty}$-hermitian metric of $\OO(1)$ given by
\[
\vert T_i \vert_{\pmb{a}} = \frac{\vert T_i \vert}{\sqrt{a_0 \vert T_1 \vert^2 + a_1 \vert T_1 \vert^2 + \cdots + a_n \vert T_n \vert^2}}
\]
for $i=0, \ldots, n$.
Then $c_1(\OO(1), \vert \cdot\vert_{\pmb{a}}) = \omega_{\pmb{a}}$.
Thus the second assertion follows.
\end{proof}

We define a function $\varphi_{\pmb{a}} : \RR_{\geq 0}^{n+1} \to \RR$ to be
\[
\varphi_{\pmb{a}}(x_0, \ldots, x_n) = -\sum_{i=0}^n x_i \log x_i + \sum_{i=0}^n x_i \log a_i,
\]
which is called the {\em characteristic function of $g_{\pmb{a}}$}. 
The function $\varphi_{\pmb{a}}$ play a key role in this paper.
Here note that $\varphi_{\pmb{a}}(0, \ldots, \stackrel{\stackrel{i}{\vee}}{1}, \ldots, 0) = \log a_i$ for $i=0, \ldots, n$.
Notably the characteristic function is very similar to the entropy function in the coding theory.

\begin{Lemma}
\label{lem:characteristic:fun:max}
For $(x_0, \ldots, x_n) \in \RR_{\geq 0}^{n+1}$ with $x_0 + x_1 + \cdots + x_n = 1$,
\[
\varphi_{\pmb{a}}(x_0, \ldots, x_n) \leq  \log(a_0 + a_1 + \cdots + a_n),
\]
and the equality holds if and only if
\[
x_0 = a_0/(a_0 + a_1 + \cdots + a_n), \ldots, x_n = a_n/(a_0 + a_1 + \cdots + a_n).
\]
\end{Lemma}

\begin{proof}
Let us begin with the following claim:

\begin{Claim}
\label{claim:lem:characteristic:fun:max}
For $\alpha_1, \ldots, \alpha_r, \beta_1, \ldots, \beta_r, t_1, \ldots, t_r \in \RR_{>0}$ with $\alpha_1 + \cdots + \alpha_r = 1$,
\addtocounter{Claim}{1}
\[
\sum_{i=1}^r \alpha_i \log t_i  \leq \log\left( \sum_{i=1}^r \beta_i t_i \right) + \sum_{i=1}^r \alpha_i \log \frac{\alpha_i}{\beta_i},
\]
and the equality holds if and only if $\frac{\beta_1}{\alpha_1} t_1  = \cdots = \frac{\beta_r}{\alpha_r} t_r$.
\end{Claim}

\begin{proof}
Note that if we set $t'_i = \frac{\beta_i}{\alpha_i}t_i$ for $i=1, \ldots, r$, then
\[
\sum_{i=1}^r \alpha_i \log t_i - \log\left( \sum_{i=1}^r \beta_i t_i \right) =
\sum_{i=1}^r \alpha_i \log t'_i  -\log \left( \sum_{i=1}^r \alpha_i t'_i \right) +
\sum_{i=1}^r \alpha_i \log \frac{\alpha_i}{\beta_i}.
\]
Thus we may assume that $\alpha_i = \beta_i$ for all $i$.
In this case, the inequality is nothing more than Jensen's inequality for the strictly concave function $\log$.
\end{proof}

We set $I = \{ i \mid x_i \not= 0 \}$. Then, using the above claim, we have
\[
\sum_{i \in I} x_i \log a_i \leq \log \left( \sum_{i \in I} a_i \right) + \sum_{i \in I} x_i \log x_i,
\]
and hence
\begin{align*}
\varphi_{\pmb{a}}(x_0, \ldots, x_n) & = \sum_{i \in I} -x_i \log x_i + \sum_{i \in I} x_i \log a_i \\
& \leq \log \left( \sum_{i \in I} a_i \right) \leq \log (a_0 + \cdots + a_n).
\end{align*}
In addition, the equality holds if and only if
$a_i/x_i = a_j/x_j$ for all $i, j \in I$ and $a_i = 0$ for all $i \not\in I$.
Thus the assertion follows.
\end{proof}

Note that 
\[
H^0(\PP^n_{\ZZ}, lH_0) = \bigoplus_{\substack{\pmb{e} \in \ZZ_{\geq 0}^{n},\vert \pmb{e} \vert \leq l} } \ZZ z^{\pmb{e}}
\]
(for the definition of $\vert \pmb{e} \vert$ and $z^{\pmb{e}}$, see Conventions and terminology~\ref{CT:entry} and \ref{CT:mult:coeff}).
According as \cite{MoArZariski}, $\vert\cdot\vert_{lg_{\pmb{a}}}$, $\Vert\cdot\Vert_{lg_{\pmb{a}}}$ and
$\langle\cdot,\cdot\rangle_{lg_{\pmb{a}}}$ are defined by
\begin{align*}
& \vert\phi\vert_{lg_{\pmb{a}}} := \vert \phi \vert \exp(-lg_{\pmb{a}}/2),\quad
\Vert\phi\Vert_{lg_{\pmb{a}}} := \sup \{ \vert\phi\vert_{lg_{\pmb{a}}}(x) \mid x \in \PP^n(\CC) \}
\intertext{and}
& \langle\phi,\psi \rangle_{lg_{\pmb{a}}} := \int_{\PP^n(\CC)} \phi \bar{\psi} \exp(-lg_{\pmb{a}}) \Phi_{\pmb{a}},
\end{align*}
where $\phi, \psi \in H^0(\PP^n(\CC), lH_0)$.

\begin{Proposition}
\label{prop:cal:inner:product}
Let $l$ be a positive integer and $\pmb{e} = (e_1, \ldots, e_n), \pmb{e}' = (e'_1, \ldots, e'_n) \in \ZZ_{\geq 0}^n$ with
$\vert \pmb{e} \vert, \vert \pmb{e}' \vert \leq l$.
\begin{enumerate}
\renewcommand{\labelenumi}{(\arabic{enumi})}
\item
$\Vert z^{\pmb{e}} \Vert_{lg_{\pmb{a}}}^2 = \exp(-l \varphi_{\pmb{a}}(\widetilde{\pmb{e}}^l/l))$
\rom{(}for the definition of $\widetilde{\pmb{e}}^l$, see Conventions and terminology~\rom{\ref{CT:lifting}}\rom{)}.

\item
\[
\langle z^{\pmb{e}}, z^{\pmb{e}'} \rangle_{lg_{\pmb{a}}} = \begin{cases}
0 & \text{if $\pmb{e} \not= \pmb{e}'$}, \\
\\
{\displaystyle \frac{1}{\binom{n+l}{n}\binom{l}{\widetilde{\pmb{e}}^l}\pmb{a}^{\widetilde{\pmb{e}}^l}}} & \text{if $\pmb{e} = \pmb{e}'$}
\end{cases}
\]
\rom{(}for the definition of $\binom{l}{\widetilde{\pmb{e}}^l}$, see Conventions and terminology~\rom{\ref{CT:mult:coeff}}\rom{)}.
\end{enumerate}
\end{Proposition}

\begin{proof}
(1) By the definition of $\vert z^{\pmb{e}} \vert_{lg_{\pmb{a}}}$, we can see
\[
\log \vert z^{\pmb{e}} \vert_{lg_{\pmb{a}}}^2 = e_0 \log \vert T_0 \vert^2 + \cdots + e_n \log \vert T_n \vert^2 - l \log(a_0 \vert T_0 \vert^2 + \cdots + a_n \vert T_n \vert^2),
\]
where $e_0 = l - e_1 - \cdots - e_n$ and $(T_0: \cdots : T_n)$ is a homogeneous coordinate of $\PP^n(\CC)$ such that $z_i = T_i/T_0$.
Here we set $e'_i = e_i/l$ for $i=0, \ldots, l$ and $I = \{ i \mid e_i \not= 0 \}$.
Then, by using Claim~\ref{claim:lem:characteristic:fun:max},
\[
\frac{1}{l} \log \vert z^{\pmb{e}} \vert_{lg_{\pmb{a}}}^2 \leq 
\sum_{i \in I} e'_i \log \vert T_i \vert^2  -  \log\left(\sum_{i \in I} a_i \vert T_i \vert^2 \right) \leq - \varphi_{\pmb{a}}(e'_0, \ldots, e'_n).
\]
Moreover, if we set $T_i = \sqrt{e'_i/a_i}$ for $i=0, \ldots, n$, then the equality holds. Thus (1) follows.

\medskip
(2) First of all, Proposition~\ref{prop:calculation:volume:form},
\[
\langle z^{\pmb{e}}, z^{\pmb{e}'} \rangle_{lg_{\pmb{a}}} = \left( \frac{\sqrt{-1}}{2\pi}\right)^n
 \int_{\PP^n(\CC)} \frac{n! a_0 \cdots a_n z^{\pmb{e}} \bar{z}^{\pmb{e}'}d z_1 \wedge d\bar{z}_1 \wedge \cdots \wedge dz_n \wedge d\bar{z}_n}{(a_0 + a_1 \vert z_1 \vert^2 + \cdots + a_n \vert z_n \vert^2)^{n+l+1}}.
\]
If we set $z_i = x_i^{1/2} \exp(2\pi\sqrt{-1}\theta_i)$, then the above integral is equal to
\[
\int_{\RR^n \times [0,1]^n} \frac{n! a_0 \cdots a_n \prod_{i=1}^n x_i^{(e_i + e_i')/2} \exp(2\pi\sqrt{-1}(e_i - e'_i))}
{(a_0 + a_1 x_1 + \cdots + a_n x_n)^{n+l+1}}dx_1 \cdots dx_n d\theta_1 \cdots d\theta_n,
\]
and hence
\[
\langle z^{\pmb{e}}, z^{\pmb{e}'} \rangle_{lg_{\pmb{a}}} =
\begin{cases}
0 & \text{if $\pmb{e} \not= \pmb{e}'$}, \\
\\
{\displaystyle \int_{\RR^n} \frac{n! a_0 \cdots a_n x_1^{e_1}\cdots x_n^{e_n}} 
{(a_0 + a_1 x_1 + \cdots + a_n x_n)^{n+l+1}}dx_1 \cdots dx_n} & \text{if $\pmb{e}=\pmb{e}'$}.
\end{cases}
\]
It is easy to see that
\[
\int_{0}^{\infty} \frac{a x^m}{(ax + b)^n} dx = \frac{m!}{a^m b^{n - m -1} (n-1) (n-2) \cdots (n-m)(n-m-1)}
\]
for $a, b \in \RR_{>0}$ and $n, m \in \ZZ_{\geq 0}$ with $n - m \geq 2$.
Thus we can see
\[
\langle z^{\pmb{e}}, z^{\pmb{e}} \rangle_{lg_{\pmb{a}}} =
\frac{n! e_n! \cdots e_1!}{(n+l)(n+l-1)\cdots (e_0 + 1) a_n^{e_n} \cdots a_1^{e_1} a_0^{e_0}},
\]
where $e_0 = l - e_1 - \cdots - e_n$.
Therefore
the assertion follows.
\end{proof}

Next we observe the following lemma:

\begin{Lemma}
\label{lem:stirling}
If we set $A_n = (n+2)/2$ and $B_n = (n+2)\log \sqrt{2\pi} + (n+2)/12$, then
\[
\left\vert \frac{1}{l} \log \left( \frac{l!}{k_0! \cdots k_n!} a_0^{k_0} \cdots a_n^{k_n} \right) - 
\varphi_{\pmb{a}}(k_0/l, \ldots, k_n/l) \right\vert
\leq \frac{1}{l} (A_n \log l + B_n)
\]
holds for all $l \geq 1$ and $(k_0, \ldots, k_n) \in \ZZ_{\geq 0}^{n+1}$ with $k_0 + \cdots + k_n = l$.
\end{Lemma}

\begin{proof}
First of all, note that, for $m \geq 1$,
\[
m! = \sqrt{2\pi m} \ \frac{m^m}{e^m} \ e^{\frac{\theta_m}{12m}}\quad(0 < \theta_m < 1)
\]
by Stirling's formula. We set $I = \{ i \mid k_i \not= 0\}$.
Then
\begin{align*}
\log (l!) &= \log(\sqrt{2\pi l}) + l \log l - l + \frac{\theta_l}{12l},\\
\log (k_i!) & = \log(\sqrt{2\pi k_i}) + k_i \log k_i - k_i + \frac{\theta_{k_i}}{12k_i}\quad(i \in I).
\end{align*}
Therefore,
\begin{multline*}
\frac{1}{l} \log \left( \frac{l!}{k_0! \cdots k_n!} a_0^{k_0} \cdots a_n^{k_n} \right) =
\varphi_{\pmb{a}}(k_0/l, \ldots, k_n/l) \\
+ \frac{1}{l} \log(\sqrt{2\pi l}) + \frac{\theta_l}{12l^2}
-\sum_{i \in I} \left( \frac{1}{l} \log(\sqrt{2\pi k_i}) + \frac{\theta_{k_i}}{12lk_i} \right),
\end{multline*}
which yields the assertion.
\end{proof}

Let $\overline{D}_{\pmb{a}}$ be an arithmetic divisor of ($C^{\infty} \cap \Tpsh$)-type on $\PP^n_{\ZZ}$ 
given by
\[
\overline{D}_{\pmb{a}} := (H_0, g_{\pmb{a}}) = (H_0, \log(a_0 + a_1 \vert z_1 \vert^2 + \cdots + a_n \vert z_n \vert^2)).
\]
Moreover, $\Theta_{\pmb{a}}$ is defined to be
\[
\Theta_{\pmb{a}} := \{ (x_1, \ldots, x_n) \in \Delta_n \mid \varphi_{\pmb{a}}(1 - x_1 - \cdots - x_n, x_1, \ldots, x_n ) \geq 0 \},
\]
where $\Delta_n = \{ (x_1, \ldots, x_n) \in \RR_{\geq 0}^{n} \mid x_1 + \cdots + x_n \leq 1 \}$.
Note that $\Theta_{\pmb{a}}$ is a compact convex set.
Finally we consider the following proposition:

\begin{Proposition}
\label{prop:criterion:QQ:effective}
Let us fix a positive integer $l$. Then we have the following:
\begin{enumerate}
\renewcommand{\labelenumi}{(\arabic{enumi})}
\item
$\aH(\PP^n_{\ZZ}, l \overline{D}_{\pmb{a}}) \not= \{ 0 \}$ if and only if
$l \Theta_{\pmb{a}} \cap \ZZ^{n} \not= \emptyset$.

\item
If $l \Theta_{\pmb{a}} \cap \ZZ \not= \emptyset$,
then $\langle \aH(\PP^n_{\ZZ}, l \overline{D}_{\pmb{a}}) \rangle_{\ZZ} = \bigoplus_{\pmb{e} \in l \Theta_{\pmb{a}} \cap \ZZ^n} \ZZ z^{\pmb{e}}$.
\end{enumerate}
\end{Proposition}

\begin{proof}
Let us begin with the following claim:

\begin{Claim}
Let $\phi \in \aH(\PP^n_{\ZZ}, l \overline{D}_{\pmb{a}})$.
If we write
\[
\phi = \sum_{\substack{\pmb{e} \in \ZZ_{\geq 0}^n, \vert \pmb{e} \vert \leq l}} c_{\pmb{e}} z^{\pmb{e}}\quad(c_{\pmb{e}} \in \ZZ),
\]
then $\{ \pmb{e} \mid c_{\pmb{e}} \not= 0 \} \subseteq l \Theta_{\pmb{a}}$.
\end{Claim}

\begin{proof}
Clearly we may assume that $\phi \not= 0$.
We set $\{ \pmb{e} \mid c_{\pmb{e}} \not= 0 \} = \{ \pmb{e}_1, \ldots, \pmb{e}_m \}$,
where $\pmb{e}_i \not= \pmb{e}_{j}$ for $i \not= j$.
Let $\pmb{e}_i$ be an extreme point of $\Conv\{ \pmb{e}_1, \ldots, \pmb{e}_m \}$.
Here let us see that
$\pmb{e}_i \in l \Theta_{\pmb{a}}$.
Renumbering $\pmb{e}_1, \ldots, \pmb{e}_m$, we may assume that $i = 1$.
Then, for $k \geq 1$,
\[
\phi^k = c_{\pmb{e}_1}^k z^{k\pmb{e}_1} + \sum_{\substack{k_1, \ldots, k_m \in \ZZ_{\geq 0},\\
k_1 + \cdots + k_m = k, \ k_1 \not= k}} \frac{k!}{k_1! \cdots k_m!} c_{\pmb{e}_1}^{k_1} \cdots c_{\pmb{e}_m}^{k_m} z^{k_1 \pmb{e}_1 + \cdots + k_m \pmb{e}_m}.
\]
Let us check that $k\pmb{e}_1 \not= k_1 \pmb{e}_1 + \cdots + k_m \pmb{e}_m$ holds for all $k_1, \ldots, k_m \in \ZZ_{\geq 0}$ with
$k_1 + \cdots + k_m = k$ and $k_1 \not= k$. Otherwise,
$\pmb{e}_1 = (k_2/(k - k_1))\pmb{e}_2 + \cdots + (k_m/(k - k_1))\pmb{e}_m$.
This is a contradiction because $\pmb{e}_1$ is an extreme point of $\Conv\{ \pmb{e}_1, \ldots, \pmb{e}_m \}$.
Therefore, we can write
\[
\phi^k = c_{\pmb{e}_1}^k z^{k\pmb{e}_1} + \sum_{\pmb{e}' \in \ZZ_{\geq 0}^n, \pmb{e}' \not= k\pmb{e}_1} c'_{\pmb{e}'} z^{\pmb{e}'}
\]
for some $c'_{\pmb{e}'} \in \ZZ$, which implies
\[
\langle \phi^k, \phi^k \rangle_{klg_{\pmb{a}}} = \frac{c_{\pmb{e}_1}^{2k}}{\binom{kl + n}{n} \binom{kl}{k\widetilde{\pmb{e}}^l_1} \pmb{a}^{k\widetilde{\pmb{e}}^l_1}} + (\text{non-negative real
number})
\]
by Proposition~\ref{prop:cal:inner:product}.
Since $\phi^k \in \aH(\PP^n_{\ZZ}, kl\overline{D}_{\pmb{a}})$, we have $\langle \phi^k, \phi^k \rangle_{klg_{\pmb{a}}} \leq 1$,
which yields 
\[
\binom{kl + n}{n} \binom{kl}{k\widetilde{\pmb{e}}^l_1} \pmb{a}^{k\widetilde{\pmb{e}}^l_1} \geq 1.
\]
Thus, by Lemma~\ref{lem:stirling},
\[
\varphi_{\pmb{a}}\left(\frac{k\widetilde{\pmb{e}}^l_1}{kl}\right) \geq -\frac{1}{kl} (A_n \log(kl) + B_n) - \frac{1}{kl} \log \binom{kl + n}{n}.
\]
Therefore, by taking $k \to \infty$, $\varphi_{\pmb{a}}\left(\frac{\widetilde{\pmb{e}}^l_1}{l}\right) \geq 0$, and hence
$\pmb{e}_1 \in l \Theta_{\pmb{a}}$.

Finally let us see the claim.
Let $\pmb{e}_{i_1}, \ldots, \pmb{e}_{i_r}$ be all extreme points of $\Conv\{ \pmb{e}_1, \ldots, \pmb{e}_m \}$.
Then, by the above observation, 
\[
\Conv\{ \pmb{e}_1, \ldots, \pmb{e}_m \} = \Conv\{ \pmb{e}_{i_1}, \ldots, \pmb{e}_{i_r} \} \subseteq l \Theta_{\pmb{a}}
\]
because $l \Theta_{\pmb{a}}$ is a convex set.
\end{proof}

Let us go back to the proofs of (1) and (2).
By Proposition~\ref{prop:cal:inner:product},
\[
\Vert z^{\pmb{e}} \Vert_{lg_{\pmb{a}}} =  \exp(-l\varphi_{\pmb{a}}(\widetilde{\pmb{e}}^l/l)).
\]
Thus
(1) and (2) follow from the above claim.
\end{proof}

\begin{Remark}
Let $\tilde{\rho}_{\pmb{a}}$ be a hermitian inner product of $H^0(\PP^n(\CC), \OO_{\PP^n}(1))$ given by
\[
\left( \tilde{\rho}_{\pmb{a}}(T_i, T_j) \right)_{0 \leq i, j \leq n} = \begin{pmatrix}
1/a_0 & 0 & \cdots & 0 & 0 \\
0 & 1/a_1 & \cdots &  0 & 0  \\
\vdots &    \vdots       &      \ddots        & \vdots & \vdots \\
0 &   0        & \cdots            & 1/a_{n-1} & 0 \\
0 & 0 & \cdots & 0 & 1/a_n 
\end{pmatrix}.
\]
Let $\rho_{\pmb{a}}$ be the quotient $C^{\infty}$-hermitian metric of $\OO_{\PP^n}(1)$ 
induced by $\tilde{\rho}_{\pmb{a}}$ and the canonical homomorphism
\[
H^0(\PP^n(\CC), \OO_{\PP^n}(1)) \otimes \OO_{\PP^n} \to \OO_{\PP^n}(1).
\]
Then $g_{\pmb{a}} = -\log \rho_{\pmb{a}}(T_0, T_0)$.
\end{Remark}

\begin{Remark}
Hajli \cite{Hajli} pointed out that, for $(x_1, \ldots, x_n) \in \Delta_n$, 
\[
-\varphi_{\pmb{a}}(1-x_1-\cdots-x_n, x_1, \ldots, x_n)
\]
is the Legendre-Fenchel transform of $\log(a_0 + a_1 e^{u_1} + \cdots + a_n e^{u_n})$, that is,
\begin{multline*}
\hspace{-1em}-\varphi_{\pmb{a}}(1-x_1-\cdots-x_n, x_1, \ldots, x_n) \\
\hspace{1.5em}= \sup \left.\left\{ u_1 x_1 + \cdots + u_n x_n - \log (a_0 + a_1 e^{u_1} + \cdots + a_n e^{u_n}) \ \right|\  (u_1, \ldots, u_n) \in \RR^n\right\}.
\end{multline*}
This can be easily checked by Claim~\ref{claim:lem:characteristic:fun:max}.
\end{Remark}

\section{Integral formula and Geography of $\overline{D}_{a,b}$}
Let $X$ be a $d$-dimensional, generically smooth, normal and projective arithmetic variety.
Let $\overline{D} = (D, g)$ be an arithmetic $\RR$-divisor of $C^0$-type on $X$.
Let $\Phi$ be an $F_{\infty}$-invariant volume form on $X(\CC)$ with ${\displaystyle \int_{X(\CC)} \Phi = 1}$.
Recall that $\langle \phi, \psi \rangle_g$ and $\Vert \phi \Vert_{g, L^2}$ are given by
\[
\langle \phi, \psi \rangle_g := \int_{X(\CC)} \phi \bar{\psi} \exp(-g) \Phi\quad\text{and}\quad
\Vert \phi \Vert_{g, L^2} := \sqrt{\langle \phi, \phi \rangle_g}
\]
for $\phi, \psi \in H^0(X, D)$.
We set
\[
\aH_{L^2}(X, \overline{D}) := \{ \phi \in H^0(X, D) \mid \Vert \phi\Vert_{g, L^2} \leq 1 \}.
\]
Let us begin with the following lemmas:

\begin{Lemma}
\label{lem:vol:L:2}
${\displaystyle \avol(\overline{D}) = \lim_{l\to\infty} \frac{\log \# \aH_{L^2}(X, l\overline{D})}{l^d/d!}}$.
\end{Lemma}

\begin{proof}
First of all, note that
\[
\avol(\overline{D}) = \lim_{l\to\infty} \frac{\log \# \aH(X, l\overline{D})}{l^d/d!}
\]
(cf. \cite[Theorem~5.2.2]{MoArZariski}).
Since $\aH(X, l\overline{D}) \subseteq \aH_{L^2}(X, l\overline{D})$, we have
\[
\avol(\overline{D}) \leq \liminf_{l\to\infty} \frac{\log \# \aH_{L^2}(X, l\overline{D})}{l^d/d!}.
\]
On the other hand, by using Gromov's inequality (cf. \cite[Proposition~3.1.1]{MoArZariski}), 
there is a constant $C$ such that $\Vert\cdot\Vert_{\sup} \leq Cl^{d-1}\Vert\cdot\Vert_{L^2}$
on $H^0(X, lD)$.
Thus, for any positive number $\epsilon$,
$\Vert\cdot\Vert_{\sup} \leq \exp(l\epsilon/2) \Vert\cdot\Vert_{L^2}$ holds for $l \gg 1$.
This implies that 
\[
\aH_{L^2}(X, l\overline{D}) \subseteq \aH(X, l(\overline{D} + (0, \epsilon)))
\]
for $l \gg 1$, which yields
\[
 \limsup_{l\to\infty} \frac{\log \# \aH_{L^2}(X, l\overline{D})}{l^d/d!} \leq \avol(\overline{D} + (0, \epsilon)).
\]
Therefore, by virtue of the continuity of $\avol$, we have
\[
 \limsup_{l\to\infty} \frac{\log \#\aH_{L^2}(X, l\overline{D})}{l^d/d!} \leq \avol(\overline{D}),
\]
and hence the lemma follows.
\end{proof}

\begin{Lemma}
\label{lem:integral:formula:lattice:points}
Let $\Theta$ be a compact convex set in $\RR^n$ such that $\operatorname{vol}(\Theta) > 0$.
For each $l \in \ZZ_{\geq 1}$, let $A_l = (a_{\pmb{e},\pmb{e}'})_{\pmb{e},\pmb{e}' \in l\Theta \cap \ZZ^n}$ be
a positive definite symmetric real matrix indexed by $l\Theta \cap \ZZ^n$, and let $K_l$ be a subset of 
$\RR^{l\Theta \cap \ZZ^n} \simeq \RR^{\#(l\Theta \cap \ZZ^n)}$ given by
\[
K_l = \left\{ (x_{\pmb{e}}) \in \RR^{l\Theta \cap \ZZ^n} \ \left|\  \sum_{\pmb{e},\pmb{e}' \in l \Theta \cap \ZZ^n} a_{\pmb{e},\pmb{e}'}x_{\pmb{e}}x_{\pmb{e}'} \leq 1 \right\}\right. .
\]
We assume that there are positive constants $C$ and $D$ and a continuous function $\varphi : \Theta \to \RR$ such that
\[
\left| \log \left(\frac{1}{a_{\pmb{e},\pmb{e}}}\right) - l \varphi\left(\frac{\pmb{e}}{l}\right) \right| \leq C \log(l) + D
\]
for all $l \in \ZZ_{\geq 1}$ and $\pmb{e} \in l\Theta \cap \ZZ^n$.
Then we have
\[
\liminf_{l\to\infty} \frac{\log \# (K_l \cap \ZZ^{l\Theta \cap \ZZ^n})}{l^{n+1}} \geq \frac{1}{2} \int_\Theta \varphi(\pmb{x}) d\pmb{x}.
\]
Moreover, if $A_l$ is diagonal and all entries of $A_l$ are less than or equal to $1$ \rom{(}i.e.,
$a_{\pmb{e},\pmb{e}'} \leq 1$ $\forall \pmb{e},\pmb{e}'  \in l\Theta \cap \ZZ^n$\rom{)} for each $l$, then
\[
\lim_{l\to\infty} \frac{\log \# (K_l \cap \ZZ^{l\Theta \cap \ZZ^n})}{l^{n+1}} = \frac{1}{2} \int_\Theta \varphi(\pmb{x}) d\pmb{x}.
\]
\end{Lemma}

\begin{proof}
By Minkowski's theorem,
\[
\log \# (K_l \cap \ZZ^{l \Theta \cap \ZZ^n}) \geq \log(\operatorname{vol}(K_l)) - m_l \log(2),
\]
where $m_l = \#(l  \Theta \cap \ZZ^n)$.
Note that
\[
 \log(\operatorname{vol}(K_l)) = - \frac{1}{2} \log (\det(A_l))  
+ \log V_{m_l},
\]
where $V_r = \operatorname{vol}( \{(x_{1}, \ldots, x_{r}) \in \RR^{r} \mid 
x_1^2 + \cdots + x_r^2 \leq 1\})$. Moreover, by Hadamard's inequality, 
\[
\det(A_l) \leq \prod_{\pmb{e} \in l \Theta \cap \ZZ^n} a_{\pmb{e},\pmb{e}}.
\]
Thus
\[
\log \# (K_l \cap \ZZ^{l\Theta \cap \ZZ^n})
\geq  \frac{1}{2} \sum_{\pmb{e} \in l\Theta \cap \ZZ^n}\log \left( \frac{1}{a_{\pmb{e},\pmb{e}}} \right)
+ \log V_{m_l} - m_l \log(2).
\]
Further, there is a positive constant $c_1$ such that $m_l \leq c_1 l^n$ for $l \geq 1$. Thus we can see
\[
\lim_{l\to\infty} \log (V_{m_l})/l^{n+1} = 0.
\]
Therefore, it is sufficient to show that
\[
\lim_{l\to\infty} \frac{1}{l^{n+1}} \sum_{\pmb{e} \in l\Theta \cap \ZZ^n}
 \log \left( \frac{1}{a_{\pmb{e},\pmb{e}}} \right) =   \int_{\Theta} \varphi(\pmb{x}) d\pmb{x}.
\]
By our assumption, we have
\[
\varphi\left(\frac{\pmb{e}}{l}\right) - \frac{1}{l}(C \log l + D) \leq
\frac{1}{l}  \log \left(  \frac{1}{a_{\pmb{e},\pmb{e}}} \right) \leq
\varphi\left(\frac{\pmb{e}}{l}\right) +  \frac{1}{l}(C \log l + D).
\]
Note that
\[
\lim_{l\to\infty} \frac{1}{l^{n}} \sum_{\pmb{e} \in l \Theta \cap \ZZ^n} \varphi\left(\frac{\pmb{e}}{l}\right) =
\lim_{l\to\infty}  \sum_{\pmb{x} \in \Theta \cap (1/l) \ZZ^n} \varphi(\pmb{x}) \frac{1}{l^{n}} =
\int_{\Theta} \varphi(\pmb{x}) d\pmb{x}.
\]
On the other hand,
since $m_l \leq c_1 l^n$, we can see
\[
 \lim_{l\to\infty} \sum_{\pmb{e} \in l\Theta \cap \ZZ^n} \frac{1}{l^{n+1}}(C \log l + D) = 0.
\]
Thus the first assertion follows.

\medskip
Next we assume that $A_l$ is diagonal for each $l$.
Then, since
\[
K_l  \subseteq  \prod_{\pmb{e} \in l\Theta \cap \ZZ^n} \left[ -\sqrt{\frac{1}{a_{\pmb{e},\pmb{e}}}},
\sqrt{\frac{1}{a_{\pmb{e},\pmb{e}}}} \ \right],
\]
we have
\[
\log \# (K_l \cap \ZZ^{l\Theta \cap \ZZ^n})
\leq \sum_{\pmb{e} \in l\Theta \cap \ZZ^n} \log \left( 2\sqrt{\frac{1}{a_{\pmb{e},\pmb{e}}}}+ 1\right).
\]
Thus
\[
\log \# (K_l \cap \ZZ^{l\Theta \cap \ZZ^n})
\leq \frac{1}{2} \sum_{\pmb{e} \in l\Theta \cap \ZZ^n}\log \left( \frac{1}{a_{\pmb{e},\pmb{e}}}\right)  + m_l\log(3)
\]
because $a_{\pmb{e},\pmb{e}} \leq 1$ and $2 t + 1 \leq 3t$ for $t \geq 1$.
Therefore, as before,
\[
\limsup_{l\to\infty} \frac{\log \# (K_l \cap \ZZ^{l\Theta \cap \ZZ^n})}{l^{n+1}} \leq \frac{1}{2} \int_\Theta \varphi(\pmb{x}) d\pmb{x}.
\]
\end{proof}

From now on,
we use the same notation as in Section~\ref{sec:fund:prop:func}.
The purpose of this section is to prove the following theorem:

\begin{Theorem}
\label{thm:positivity:D:a:b}
\begin{enumerate}
\renewcommand{\labelenumi}{(\arabic{enumi})}
\item \rom{(}Integral formula\rom{)}\quad
The following formulae hold:
\begin{align*}
\hspace{4em}&\avol(\overline{D}_{\pmb{a}}) = \frac{(n+1)!}{2} \int_{\Theta_{\pmb{a}}} \varphi_{\pmb{a}}(1-x_1 - \cdots - x_n, x_1, \ldots, x_n) 
dx_1 \cdots dx_n, \\
\intertext{and}
& \adeg(\overline{D}_{\pmb{a}}^{n+1}) = \frac{(n+1)!}{2} \int_{\Delta_n} \varphi_{\pmb{a}}(1-x_1 - \cdots - x_n, x_1, \ldots, x_n) 
dx_1 \cdots dx_n.
\end{align*}

\item
$\overline{D}_{\pmb{a}}$ is ample if and only if $\pmb{a}(i) > 1$ for all $i=0, \ldots, n$.

\item
$\overline{D}_{\pmb{a}}$ is nef if and only if $\pmb{a}(i) \geq 1$ for all $i=0, \ldots, n$.

\item
$\overline{D}_{\pmb{a}}$ is big if and only if $\vert \pmb{a} \vert > 1$.

\item
$\overline{D}_{\pmb{a}}$ is pseudo-effective if and only if $\vert \pmb{a} \vert \geq 1$.

\item
If $\vert \pmb{a} \vert = 1$, then
\[
\aH(\PP^n_{\ZZ}, l\overline{D}_{\pmb{a}}) =
\begin{cases}
\{ 0, \pm z_1^{l\pmb{a}(1)} \cdots z_n^{l\pmb{a}(n)} \} & \text{if $l\pmb{a} \in \ZZ$}, \\
\{ 0 \} & \text{if $l\pmb{a} \not\in \ZZ$}.
\end{cases}
\] 

\item
$\adeg(\overline{D}_{\pmb{a}}^{n+1}) = \avol(\overline{D}_{\pmb{a}})$ if and only if
$\overline{D}_{\pmb{a}}$ is nef.
\end{enumerate}
\end{Theorem}

\begin{proof}
First let us see the essential case of (1):

\begin{Claim}
\label{claim:thm:positivity:D:a:b:1}
If $\vert \pmb{a} \vert  > 1$, then ${\displaystyle \avol(\overline{D}_{\pmb{a}}) = \frac{(n+1)!}{2} \int_{\Theta_{\pmb{a}}} \varphi_{\pmb{a}}(\widetilde{\pmb{t}}) d\pmb{t}}$.
\end{Claim}

\begin{proof}
In this case, $\operatorname{vol}(\Theta_{\pmb{a}}) > 0$.
By using Proposition~\ref{prop:criterion:QQ:effective}, 
\[
\aH(\PP^n_{\ZZ}, l\overline{D}_{\pmb{a}}) \subseteq \left.\left\{ \phi \in \bigoplus_{\pmb{e} \in l  \Theta_{\pmb{a}} \cap \ZZ^n} \ZZ z^{\pmb{e}} \ \right| \ 
\langle \phi, \phi \rangle_{lg_{\pmb{a}}} \leq 1 \right\}
\subseteq  \aH_{L^2}(\PP^n_{\ZZ}, l\overline{D}_{\pmb{a}}),
\]
which yields
\[
\avol(\overline{D}_{\pmb{a}}) = (n+1)!  \lim_{l\to\infty} \frac{\log \#\left.\left\{ \phi \in \bigoplus_{\pmb{e} \in l  \Theta_{\pmb{a}} \cap \ZZ^n}
 \ZZ z^{\pmb{e}} \ \right| \ 
\langle \phi, \phi \rangle_{lg_{\pmb{e}}} \leq 1 \right\}}{l^{n+1}}
\]
by Lemma~\ref{lem:vol:L:2}.
We set
\[
K_l = \left\{ (x_{\pmb{e}}) \in \RR^{l  \Theta_{\pmb{a}} \cap \ZZ^n}\ \left| \  
\sum_{\pmb{e}  \in l  \Theta_{\pmb{a}} \cap \ZZ^n} \frac{x_{\pmb{e}}^2}{\binom{l+n}{n} \binom{l}{\widetilde{\pmb{e}}^l}\pmb{a}^{\widetilde{\pmb{e}}^l}} \leq 1 \right\}\right. .
\]
Then, by Proposition~\ref{prop:cal:inner:product},
\[
\#  \left.\left\{ \phi \in \bigoplus_{\pmb{e}  \in l  \Theta_{\pmb{a}} \cap \ZZ^n} \ZZ z^{\pmb{e}} \ \right| \ 
\langle \phi, \phi \rangle_{lg_{\pmb{a}}} \leq 1 \right\}
= \# (K_l \cap \ZZ^{l  \Theta_{\pmb{a}} \cap \ZZ^n}).
\]
On the other hand,
for $\pmb{e} \in l  \Theta_{\pmb{a}} \cap \ZZ^n$,
\[
\binom{l+n}{n} \binom{l}{\widetilde{\pmb{e}}^l}\pmb{a}^{\widetilde{\pmb{e}}^l}  = \frac{1}{\langle z^{\pmb{e}}, z^{\pmb{e}} \rangle_{lg_{\pmb{a}}}} \geq \exp(l\varphi_{\pmb{a}}(\widetilde{\pmb{e}}^l/l)) \geq 1.
\]
Moreover, by Lemma~\ref{lem:stirling}, there are positive constants $A$ and $B$ such that
\[
\left| \log \left(  \binom{l+n}{n} \binom{l}{\widetilde{\pmb{e}}^l}\pmb{a}^{\widetilde{\pmb{e}}^l} \right) - l \varphi_{\pmb{a}}(\widetilde{\pmb{e}}^l/l) \right|
\leq A \log l + B
\]
holds for all $l \in \ZZ_{\geq 1}$ and $\pmb{e}\in l  \Theta_{\pmb{a}} \cap \ZZ^n$.
Thus the assertion follows from Lemma~\ref{lem:integral:formula:lattice:points}.
\end{proof}

Next let us see the following claim:
\begin{Claim}
\label{claim:thm:positivity:D:a:b:2}
If $s, t \in \RR_{>0}$ and $\alpha, \beta \in \RR$ with $\alpha + \beta \not= 0$,
then
\[
\alpha \overline{D}_{t\pmb{a}} + \beta \overline{D}_{s\pmb{a}} = (\alpha + \beta)\overline{D}_{(t^{\alpha}s^{\beta})^{\frac{1}{\alpha + \beta}}\pmb{a}}.
\]
\end{Claim}

\begin{proof}
This is a straightforward calculation.
\end{proof}

\medskip
(2) and (3):\quad
First of all, $\omega_{\pmb{a}}$ is positive by Proposition~\ref{prop:calculation:volume:form}.
Let $\gamma_i$ be a $1$-dimensional closed subscheme given by
$H_0 \cap \cdots \cap H_{i-1} \cap H_{i+1} \cap \cdots \cap H_n$.
Then it is easy to see that $\adeg(\rest{\overline{D}_{\pmb{a}}}{\gamma_i}) = (1/2)\log(\pmb{a}(i))$.
Therefore we have ``only if'' part of (1) and (2).

We assume that $\pmb{a}(i) > 1$ for all $i$.
Then $\varphi_{\pmb{a}}$ is positive on 
\[
\{ (x_0, \ldots, x_n) \in \RR_{\geq 0}^{n+1} \mid x_0 + \cdots + x_n = 1 \}.
\]
Thus, for $\pmb{e} \in \ZZ_{\geq 0}^n$ with $\vert \pmb{e} \vert \leq 1$,
$z^{\pmb{e}}$ is a strictly small section by Proposition~\ref{prop:cal:inner:product}, which shows that
$\overline{D}_{\pmb{a}}$ is ample.

Next we assume that $\pmb{a}(i) \geq 1$ for all $i$.
Let $\gamma$ be a $1$-dimensional closed integral subscheme of $\PP^n_{\ZZ}$.
Then we can find $H_i$ such that $\gamma \not\subseteq H_i$.
Note that
\[
\overline{D}_{\pmb{a}} + \widehat{(z_i)} =
(H_i, \log(\pmb{a}(0) \vert w_0 \vert^2 + \cdots + \pmb{a}(n) \vert w_n \vert^2)),
\]
where $w_k = T_k/T_i$ ($k=0, \ldots, n$).
Therefore $\adeg(\rest{\overline{D}_{\pmb{a}}}{\gamma}) \geq 0$ because 
\[
\log(\pmb{a}(0) \vert w_0 \vert^2 + \cdots + \pmb{a}(n) \vert w_n \vert^2) \geq 0.
\]

\medskip
(6):\quad
In this case, $\Theta_{\pmb{a}} = \{ (\pmb{a}(1), \ldots, \pmb{a}(n)) \}$ and $\varphi_{\pmb{a}}(\pmb{a}) =  0$ by Lemma~\ref{lem:characteristic:fun:max}.
Moreover, if $l \pmb{a} \in \ZZ^{n+1}$, then
\[
\Vert z^{l(\pmb{a}(1), \ldots, \pmb{a}(n))} \Vert_{lg_{\pmb{a}}}^2 = \exp(-l \varphi_{\pmb{a}}(\pmb{a})) = 1
\]
by Proposition~\ref{prop:cal:inner:product}.
Thus the assertion follows from Proposition~\ref{prop:criterion:QQ:effective}.

\medskip
(4) and (5):\quad
By using (6), in order to see (4) and (5), it is sufficient to show the following:

\begin{enumerate}
\renewcommand{\labelenumi}{(\roman{enumi})}
\item
$\overline{D}_{\pmb{a}}$ is big if $\vert \pmb{a} \vert > 1$.

\item
$\overline{D}_{\pmb{a}}$ is pseudo-effective if $\vert \pmb{a} \vert \geq 1$.

\item
$\overline{D}_{\pmb{a}}$ is not pseudo-effective if $\vert \pmb{a} \vert < 1$.
\end{enumerate}

\medskip
(i) It follows from Claim~\ref{claim:thm:positivity:D:a:b:1} because $\operatorname{vol}(\Theta_{\pmb{a}}) > 0$.

\medskip
(ii) We choose a real number $t$ such that $t > 1$ and $\overline{D}_{t\pmb{a}}$ is ample.
By Claim~\ref{claim:thm:positivity:D:a:b:2},
\[
\overline{D}_{\pmb{a}} + \epsilon \overline{D}_{t\pmb{a}} = (1 + \epsilon) \overline{D}_{t^{\frac{\epsilon}{1 + \epsilon}}\pmb{a}}.
\]
For any $\epsilon > 0$,
since $t^{\frac{\epsilon}{1 + \epsilon}}\vert \pmb{a} \vert > 1$, 
$(1 + \epsilon) \overline{D}_{t^{\frac{\epsilon}{1 + \epsilon}}\pmb{a}}$ is big by (i),
which shows that $\overline{D}_{\pmb{a}}$ is pseudo-effective.

\medskip
(iii) 
Let us choose a positive real number $t$ such that $\overline{D}_{t\pmb{a}}$ is ample.
We also choose a positive number $\epsilon$ such that if we set $\pmb{a}' =  t^{\frac{\epsilon}{1 + \epsilon}}\pmb{a}$,
then $\vert \pmb{a}'\vert < 1$.
We assume that $\overline{D}_{\pmb{a}}$ is pseudo-effective. Then 
\[
\overline{D}_{\pmb{a}} + \epsilon \overline{D}_{t\pmb{a}} = (1 + \epsilon) \overline{D}_{\pmb{a}'}
\]
is big
by \cite[Proposition~6.3.2]{MoArZariski}, which means that $\overline{D}_{\pmb{a}'}$ is big.
On the other hand, as $\vert \pmb{a}'\vert < 1$, we have
$\Theta_{\pmb{a}'} = \emptyset$.
Thus $\aH(\PP^n_{\ZZ}, n\overline{D}_{\pmb{a}'}) = \{ 0 \}$
for all $n \geq 1$ by Proposition~\ref{prop:criterion:QQ:effective}. This is a contradiction.

\bigskip
(1): For the first formula, 
we may assume that $\vert \pmb{a} \vert \leq 1$ by Claim~\ref{claim:thm:positivity:D:a:b:1}.
In this case,
$\overline{D}_{\pmb{a}}$ is not big by (4) and
$\Theta_{\pmb{a}}$ is either $\emptyset$ or $\{ (a_1, \ldots, a_n) \}$.
Thus the assertion follows.
For the second formula,
the arithmetic Hilbert-Samuel formula (cf. \cite{GSRR} and \cite{AbBo}) yields
\[
\frac{\adeg(\overline{D}_{\pmb{a}}^{n+1})}{(n+1)!} =
\lim_{l\to\infty} \frac{\widehat{\chi}\left(H^0(\PP^n_{\ZZ}, l H_0), \langle\ ,\ \rangle_{lg_a} \right)}{l^{n+1}}.
\]
On the other hand,
\[
\widehat{\chi}\left(H^0(\PP^n_{\ZZ}, l H_0), \langle\ ,\ \rangle_{lg_a} \right) = 
\sum_{\pmb{e} \in l  \Delta_n \cap \ZZ^n}\log \left( \sqrt{\binom{l+n}{n} \binom{l}{\widetilde{\pmb{e}}^l}\pmb{a}^{\widetilde{\pmb{e}}^l}}\right)  + \log V_{\#(l  \Delta_n \cap \ZZ^n)}.
\]
Thus, in the same way as the proof of Lemma~\ref{lem:integral:formula:lattice:points} and
Claim~\ref{claim:thm:positivity:D:a:b:1}, we can see the second formula.

\medskip
(7): It follows from (1) and (3).
\end{proof}

\bigskip
Finally let us consider the following proposition:

\begin{Proposition}
For any positive integer $l$, there exists $\pmb{a} \in \QQ_{>0}^{n+1}$ such that
$\vert \pmb{a} \vert > 1$ and that
$\aH(\PP^n_{\ZZ}, k \overline{D}_{\pmb{a}}) = \{ 0 \}$ for $k=1, \ldots, l$.
\end{Proposition}

\begin{proof}
Let us choose positive rational numbers $a'_1, \ldots, a'_n$ such that
$a'_1 + \cdots + a'_n < 1$ and $a'_1 < 1/l, \ldots, a'_n < 1/l$.
We set $a'_0 = 1 - a'_1 - \cdots - a'_n$ and $\pmb{a}' = (a'_0, \ldots, a'_n)$.
Moreover, for a rational number $\lambda > 1$, we set
\[
K_{\lambda} = \{ \pmb{x} \in \Delta_n \mid \varphi_{\pmb{a}'}(\widetilde{\pmb{x}}) + \log \lambda \geq 0 \},
\]
where $\Delta_n = \{(x_1, \ldots, x_n) \in \RR_{\geq 0}^n \mid x_1 + \cdots + x_n \leq 1 \}$.

\begin{Claim}
We can find a rational number $\lambda > 1$ such that
$K_{\lambda} \subseteq (0,1/l)^n$.
\end{Claim}

\begin{proof}
We assume that $K_{1+(1/m)} \not\subseteq (0,1/l)^n$ for all $m \in \ZZ_{\geq 1}$, that is,
we can find $\pmb{x}_m \in K_{1+(1/m)} \setminus (0,1/l)^n$ for each $m \geq 1$.
Since $\Delta_n$ is compact, there is a subsequence $\{ \pmb{x}_{m_i} \}$ of $\{ \pmb{x}_m \}$ such that
$\pmb{x} = \lim_{i\to\infty} \pmb{x}_{m_i}$ exists. Note that $\pmb{x} \not\in (0,1/l)^n$ because $\pmb{x}_{m_i} \not\in (0,1/l)^n$ for all $i$.
On the other hand, since $\varphi_{\pmb{a}'}(\widetilde{\pmb{x}}_{m_i}) + \log(1 + (1/m_i)) \geq 0$ for all $i$,
we have $\varphi_{\pmb{a}'}(\widetilde{\pmb{x}}) \geq 0$, and hence $\pmb{x} = (a'_1, \ldots, a'_n)$ by Lemma~\ref{lem:characteristic:fun:max}.
This is a contradiction.
\end{proof}

We choose a rational number $\lambda > 1$ as in the above claim.
Here we set $\pmb{a} = \lambda \pmb{a}'$. Then, as $\varphi_{\pmb{a}} = \varphi_{\pmb{a}'} + \log \lambda$,
we have $\Theta_{\pmb{a}} \subseteq (0,1/l)^n$.
We assume that $\aH(\PP^n_{\ZZ}, k \overline{D}_{\pmb{a}}) \not= \{ 0 \}$ for some $k$ with $1 \leq k \leq l$.
Then, by Proposition~\ref{prop:criterion:QQ:effective}, there is $\pmb{e}= (e_1, \ldots, e_n) \in k \Theta_{\pmb{a}} \cap \ZZ^n$, that is,
$\pmb{e}/k \in \Theta_{\pmb{a}}$. Thus $0 < e_i/k < 1/l$ for all $i$. This is a contradiction.
\end{proof}

\section{Asymptotic multiplicity}
Let $X$ be a $d$-dimensional, projective, generically smooth and normal arithmetic variety.
Let $\overline{D}$ be an arithmetic $\RR$-divisor of $C^0$-type on $X$.
We set
\[
N(\overline{D}) = \left\{ l \in \ZZ_{>0} \mid \aH(X, l\overline{D}) \not= \{ 0 \} \right\}.
\]
We assume that $N(\overline{D}) \not= \emptyset$.
Then $\mu_x(\overline{D})$ for $x \in X$ is defined to be
\[
\mu_x(\overline{D}) := \inf \left\{ \mult_x(D + (1/l)(\phi)) \mid l \in N(\overline{D}),\ \phi \in \aH(X, l\overline{D}) \setminus \{ 0 \} \right\},
\]
which is called the {\em asymptotic multiplicity of $\overline{D}$ at $x$}.
This definition is slightly different from the way in \cite[Subsection~6.5]{MoArZariski}, but
they give the same value if $\ah(X, \overline{D}) \not= 0$ (cf. Claim~\ref{claim:prop:mu:properties:1}).

\begin{Proposition}
\label{prop:mu:properties}
Let $\overline{D}$ and $\overline{E}$ be arithmetic $\RR$-divisors of $C^0$-type such that
$N(\overline{D}) \not= \emptyset$ and $N(\overline{E}) \not= \emptyset$.
Then we have the following:
\begin{enumerate}
\renewcommand{\labelenumi}{(\arabic{enumi})}
\item
$\mu_x(\overline{D} + \overline{E}) \leq \mu_x(\overline{D}) + \mu_x(\overline{E})$.

\item
If $\overline{D} \leq \overline{E}$, then $\mu_x(\overline{E}) \leq \mu_x(\overline{D}) + \mult_{x}(E - D)$.

\item
$\mu_x(\overline{D} + \widehat{(\phi)}) = \mu_x(\overline{D})$ for $\phi \in \Rat(X)^{\times}$.

\item
$\mu_x(a\overline{D}) = a \mu_x(\overline{D})$ for $a \in \QQ_{>0}$.

\item
If $\overline{D}$ is nef and big, then $\mu_x(\overline{D}) = 0$.
\end{enumerate}
\end{Proposition}

\begin{proof}
Let us begin with the following claim:

\begin{Claim}
\label{claim:prop:mu:properties:1}
We assume that $\ah(X, \overline{D}) \not= 0$. As in \cite[Subsection~6.5]{MoArZariski}, we define $\nu_x(l\overline{D})$ 
\rom{(}$l \in \ZZ_{>0}$\rom{)} and
$\mu'_x(\overline{D})$ to be
\[
\begin{cases}
\nu_x(l\overline{D}) := \min \{ \mult_x(lD + (\phi)) \mid \phi \in \aH(X, l\overline{D}) \setminus \{ 0 \} \},\\
\mu'_x(\overline{D}) := \inf \left.\left\{ \frac{\nu_x(l\overline{D})}{l} \ \right|\  l \in \ZZ_{>0} \right\} =
\lim\limits_{l\to\infty} \frac{\nu_x(l\overline{D})}{l}.
\end{cases}
\]
Then $\mu'_x(\overline{D}) = \mu_x(\overline{D})$.
\end{Claim}

\begin{proof}
If we choose $l \in \ZZ_{>0}$ and $\phi \in \aH(l\overline{D}) \setminus \{ 0 \}$, then
\[
\mu'_x(\overline{D}) \leq  \frac{\nu_x(l\overline{D})}{l} \leq  \mult_x(D + (1/l)(\phi)),
\]
which implies
$\mu'_x(\overline{D}) \leq \mu_x(\overline{D})$.

Conversely, for each $l \in \ZZ_{>0}$,
we choose $\psi_l \in \aH(X, l\overline{D}) \setminus \{ 0 \}$ such that
$\nu_x(l\overline{D}) = \mult_x(lD + (\psi_l))$.
Then
\[
\mu_x(\overline{D}) \leq \mult_x(D + (1/l)(\psi_l)) = \frac{\nu_x(l\overline{D})}{l},
\]
and hence $\mu_x(\overline{D}) \leq \mu'_x(\overline{D})$.
Thus the claim follows.
\end{proof}

Since (1), (2), (3) and (5) follows from 
\cite[Proposition~6.5.2 and Proposition~6.5.3]{MoArZariski}, (4) and Claim~\ref{claim:prop:mu:properties:1},
it is sufficient to see (4). 

First we assume that $a \in \ZZ_{>0}$.
Let $l \in N(\overline{D})$ and $\phi \in \aH(l\overline{D}) \setminus \{ 0 \}$.
Then $\phi^a \in \aH(l(a \overline{D})) \setminus \{ 0 \}$. Thus
\[
\mu_{x}(a\overline{D}) \leq \mult_{x} (aD + (1/l)(\phi^a)) = a \mult_{x} (D + (1/l)(\phi)),
\]
which yields $\mu_{x}(a\overline{D}) \leq a \mu_{x}(\overline{D})$.
Conversely let $l \in N(a\overline{D})$ and $\psi \in \aH(l(a\overline{D})) \setminus \{ 0 \}$.
Then
\[
\mu_x(\overline{D}) \leq \mult_{x} (D + (1/la)(\psi)) = (1/a) \mult_{x} (a D + (1/l)(\psi)),
\]
and hence $\mu_x(\overline{D}) \leq (1/a) \mu_{x}(a \overline{D})$.
In general, we choose a positive integer $m$ such that $ma \in \ZZ_{>0}$. Then, by the previous observation,
\[
m \mu_{x}(a \overline{D}) = \mu_x(ma\overline{D}) = ma \mu_x(\overline{D}),
\]
as required.
\end{proof}

\begin{Lemma}
\label{lem:generator:mu}
For each $l \in N(\overline{D})$, let $\{ \phi_{l,1}, \ldots, \phi_{l,r_l} \}$ be a subset
of $\aH(X, l\overline{D}) \setminus \{ 0 \}$ such that 
$\aH(X, l\overline{D}) \subseteq \langle \phi_{l,1}, \ldots, \phi_{l,r_l} \rangle_{\ZZ}$.
Let $x$ be a point of $X$ such that the Zariski closure $\overline{ \{ x \} }$ of $\{ x \}$ is flat over $\ZZ$.
Then
\[
\mu_x(\overline{D}) = \inf \{ \mult_x\left( D + (1/l)(\phi_{l,i}) \right) \mid l \in N(D),\ i = 1, \ldots, r_l \}.
\]
\end{Lemma}

\begin{proof}
Clearly
\[
\mu_x(\overline{D}) \leq \inf \{ \mult_x\left( D + (1/l)(\phi_{l,i}) \right) \mid l \in N(D),\ i = 1, \ldots, r_l \}.
\]
Let us consider the converse inequality.
For $l \in N(\overline{D})$ and $\phi \in \aH(X, l\overline{D}) \setminus \{ 0 \}$,
we set $\phi = \sum_{i=1}^{r_l} c_i \phi_{l,i}$ for some $c_{1}, \ldots, c_{r_l} \in \ZZ$.
Note that 
\[
\mult_x((\phi + \psi)) \geq \min \{ \mult_x((\phi)), \mult_x((\psi)) \}\quad\text{and}\quad
\mult_x ((a)) = 0
\]
for $\phi, \psi \in \Rat(X)^{\times}$ and $a \in \QQ^{\times}$ with
$\phi + \psi \not= 0$.
Thus we can find $i$ such that 
\[
\mult_x((\phi)) \geq \mult_{x}((\phi_{l,i})),
\]
and hence the converse inequality holds.
\end{proof}

\section{Zariski decomposition of $\overline{D}_{\pmb{a}}$ on $\PP^1_{\ZZ}$}
We use the same notation as in Section~\ref{sec:fund:prop:func}. We assume $n=1$.
In this section, we consider the Zariski decomposition of $\overline{D}_{\pmb{a}}$ on $\PP^1_{\ZZ}$.
Note that $\Theta_{\pmb{a}}$ is a closed interval in $[0,1]$.
For simplicity, we denote the affine coordinate $z_1$ by $z$, that is, $z = T_1/T_0$.

\begin{Theorem}
\label{thm:zariski:decomp:PP:1}
The Zariski decomposition of $\overline{D}_{\pmb{a}}$ exists if and only if $a_0 + a_1 \geq 1$.
Moreover, if we set 
$\vartheta_{\pmb{a}} = \inf \Theta_{\pmb{a}}$, $\theta_{\pmb{a}} = \sup \Theta_{\pmb{a}}$,
$P_{\pmb{a}} = \theta_{\pmb{a}}H_0 - \vartheta_{\pmb{a}}H_{1}$ and
\[
p_{\pmb{a}}(z) =\begin{cases}
\vartheta_{\pmb{a}} \log \vert z \vert^2 & \text{if $\vert z \vert < \sqrt{\frac{a_0\vartheta_{\pmb{a}}}{a_1(1-\vartheta_{\pmb{a}})}}$}, \\
\log (a_0   + a_1\vert z \vert^2) & 
\text{if $\sqrt{\frac{a_0\vartheta_{\pmb{a}}}{a_1(1-\vartheta_{\pmb{a}})}} \leq \vert z \vert \leq \sqrt{\frac{a_0\theta_{\pmb{a}}}{a_1(1-\theta_{\pmb{a}})}}$}, \\
\theta_{\pmb{a}} \log \vert z \vert^2 & \text{if $\vert z \vert > \sqrt{\frac{a_0\theta_{\pmb{a}}}{a_1(1-\theta_{\pmb{a}})}}$},
\end{cases}
\]
then the positive part of $\overline{D}_{\pmb{a}}$ is $\overline{P}_{\pmb{a}} = (P_{\pmb{a}}, p_{\pmb{a}})$, where
$\sqrt{\frac{a_0\theta_{\pmb{a}}}{a_1(1-\theta_{\pmb{a}})}}$ is treated as $\infty$ if $\theta_{\pmb{a}} = 1$.
\end{Theorem}

\begin{proof}
First we consider the case where $\overline{D}_{\pmb{a}}$ is big, that is, $a_0 + a_1 > 1$ by Theorem~\ref{thm:positivity:D:a:b}.
In this case, $0 \leq \vartheta_{\pmb{a}} < \theta_{\pmb{a}} \leq 1$.
The existence of the Zariski decomposition follows from \cite[Theorem~9.2.1]{MoArZariski}.
Here we consider functions 
\begin{align*}
& r_1 : \left\{ z \in \PP^1(\CC) \ \left|\  \vert z \vert < \sqrt{\frac{a_0\theta_{\pmb{a}}}{a_1(1-\theta_{\pmb{a}})}} \right\}\right. \to \RR \\
\intertext{and}
& r_2 : \left\{ z \in \PP^1(\CC) \ \left|\  \vert z \vert > \sqrt{\frac{a_0\vartheta_{\pmb{a}}}{a_1(1-\vartheta_{\pmb{a}})}} \right\}\right. \to \RR
\end{align*}
given by
\[
r_1(z) = \begin{cases}
0  & \text{if $\vert z \vert < \sqrt{\frac{a_0\vartheta_{\pmb{a}}}{a_1(1-\vartheta_{\pmb{a}})}}$}, \\
- \vartheta_{\pmb{a}} \log \vert z \vert^2 + \log (a_0  + a_1\vert z \vert^2) & 
\text{if $\sqrt{\frac{a_0\vartheta_{\pmb{a}}}{a_1(1-\vartheta_{\pmb{a}})}} \leq \vert z \vert < \sqrt{\frac{a_0\theta_{\pmb{a}}}{a_1(1-\theta_{\pmb{a}})}}$}. \\
\end{cases}
\]
and
\[
r_2(z) = \begin{cases}
- \theta_{\pmb{a}} \log \vert z \vert^2 + \log (a_0 + a_1 \vert z \vert^2) & 
\text{if $\sqrt{\frac{a_0\vartheta_{\pmb{a}}}{a_1(1-\vartheta_{\pmb{a}})}} < \vert z \vert \leq \sqrt{\frac{a_0\theta_{\pmb{a}}}{a_1(1-\theta_{\pmb{a}})}}$}, \\
0 & \text{if $\vert z \vert > \sqrt{\frac{a_0\theta_{\pmb{a}}}{a_1(1-\theta_{\pmb{a}})}}$}.
\end{cases}
\]
In order to see that $p_{\pmb{a}}$ is a $P_{\pmb{a}}$-Green function of $(C^0 \cap \Tpsh)$-type on $\PP^1(\CC)$,
it is sufficient to check that $r_1$ and $r_2$ are  continuous and subharmonic on each area.
Let us see that $r_1$ is continuous and subharmonic. If $\vartheta_{\pmb{a}} = 0$, then the assertion is obvious, so that we may assume that
$\vartheta_{\pmb{a}} > 0$.
First of all, as $\varphi_{\pmb{a}}(1-\vartheta_{\pmb{a}}, \vartheta_{\pmb{a}}) = 0$, we have $r_1(z) = 0$ if $\vert z \vert = \sqrt{\frac{a_0\vartheta_{\pmb{a}}}{a_1(1-\vartheta_{\pmb{a}})}}$,  and hence
$r_1$ is continuous. It is obvious that $r_1$ is subharmonic on 
\[
\left\{ z \in \CC \ \left| \ \vert z \vert <\sqrt{\frac{a_0\vartheta_{\pmb{a}}}{a_1(1-\vartheta_{\pmb{a}})}}\right\}\right. \cup
\left\{ z \in \CC \ \left| \   \sqrt{\frac{a_0\vartheta_{\pmb{a}}}{a_1(1-\vartheta_{\pmb{a}})}} <  \vert z \vert < \sqrt{\frac{a_0\theta_{\pmb{a}}}{a_1(1-\theta_{\pmb{a}})}} \right\}\right..
\]
By using Claim~\ref{claim:lem:characteristic:fun:max},
\begin{align*}
 \vartheta_{\pmb{a}} \log \vert z \vert^2 & = (1-  \vartheta_{\pmb{a}})\log (1) +  \vartheta_{\pmb{a}} \log \vert z \vert^2  \\
 & \leq
 \log(a_0 + a_1 \vert z \vert^2 ) + \varphi_{\pmb{a}}(1-\vartheta_{\pmb{a}}, \vartheta_{\pmb{a}}) = \log(a_0 + a_1 \vert z \vert^2 ).
\end{align*}
Thus $r_1 \geq 0$. Therefore, if $\vert z \vert = \sqrt{\frac{a_0\vartheta_{\pmb{a}}}{a_1(1-\vartheta_{\pmb{a}})}}$, then
\[
r_1(z) = 0 \leq \frac{1}{2\pi} \int_0^{2\pi} r_1(z + \epsilon e^{\sqrt{-1}t}) dt
\]
for a small positive real number $\epsilon$, and hence $r_1$ is subharmonic.
In the similar way, we can check that $r_2$ is continuous and subharmonic.

Next let us see that $\overline{P}_{\pmb{a}}$ is nef.
As $r_1(0) = 0$ and $r_2(\infty) = 0$, we have 
\[
\adeg(\rest{\overline{P}_{\pmb{a}}}{H_0}) = \adeg(\rest{\overline{P}_{\pmb{a}}}{H_{1}}) = 0.
\]
Note that 
\[
\overline{P}_{\pmb{a}} + \vartheta_{\pmb{a}} \widehat{(z)} = ((\theta_{\pmb{a}} - \vartheta_{\pmb{a}})H_{0}, p_{\pmb{a}}(z) - \vartheta_{\pmb{a}}\log\vert z \vert^2)
\]
and
\[
p_{\pmb{a}}(z) -  \vartheta_{\pmb{a}}\log\vert z \vert^2
= 
\begin{cases}
 r_1(z) & \text{if $\vert z \vert \leq \sqrt{\frac{a_0\theta_{\pmb{a}}}{a_1(1-\theta_{\pmb{a}})}}$}, \\
(\theta_{\pmb{a}} - \vartheta_{\pmb{a}}) \log \vert z \vert^2 & \text{if $\vert z \vert > \sqrt{\frac{a_0\theta_{\pmb{a}}}{a_1(1-\theta_{\pmb{a}})}}$}.
\end{cases}
\]
Therefore, $p_{\pmb{a}}(z) -  \vartheta_{\pmb{a}}\log\vert z \vert^2 \geq 0$ on $\PP^1(\CC)$, which means that
$\overline{P}_{\pmb{a}} + \vartheta_{\pmb{a}} \widehat{(z)}$ is effective.
Let $\gamma$ be a $1$-dimensional closed integral subscheme of $\PP^1_{\ZZ}$ with
$\gamma \not= H_0, H_{1}$.
Then
\[
\adeg(\rest{\overline{P}_{\pmb{a}}}{\gamma}) = \adeg(\rest{((\theta_{\pmb{a}} - \vartheta_{\pmb{a}})H_{0}, p_{\pmb{a}} - \vartheta_{\pmb{a}}\log\vert z \vert^2)}{\gamma}) \geq 0.
\]

By using Proposition~\ref{prop:criterion:QQ:effective},
we have
$\mu_{H_0}(\overline{D}_{\pmb{a}}) = 1 - \theta_{\pmb{a}}$ and $\mu_{H_{1}}(\overline{D}_{\pmb{a}}) = \vartheta_{\pmb{a}}$.
Thus
the positive part of $\overline{D}_{\pmb{a}}$
can be written by a form 
$(P_{\pmb{a}}, q)$,
where
$q$ is a $P_{\pmb{a}}$-Green function of $(C^0 \cap \Tpsh)$-type on $\PP^1(\CC)$ (cf. \cite[Claim~9.3.5.1 and Proposition~9.3.1]{MoArZariski}).
Note that $\overline{P}_{\pmb{a}}$ is nef and $\overline{P}_{\pmb{a}} \leq \overline{D}_{\pmb{a}}$, so that 
\[
p_{\pmb{a}}(z) \leq q(z) \leq  \log(a_0 + a_1 \vert z \vert^2 ).
\]
We choose a continuous function $u$ such that $p_{\pmb{a}} + u = q$.
Then $u(z) = 0$ on 
\[
\sqrt{\frac{a_0\vartheta_{\pmb{a}}}{a_1(1-\vartheta_{\pmb{a}})}} \leq \vert z \vert \leq \sqrt{\frac{a_0\theta_{\pmb{a}}}{a_1(1-\theta_{\pmb{a}})}}.
\]
Moreover, since $q(z) = \vartheta_{\pmb{a}} \log \vert z \vert^2 + u(z)$ on 
$\vert z \vert \leq \sqrt{\frac{a_0\vartheta_{\pmb{a}}}{a_1(1-\vartheta_{\pmb{a}})}}$,
$u$ is subharmonic on $\vert z \vert \leq \sqrt{\frac{a_0\vartheta_{\pmb{a}}}{a_1(1-\vartheta_{\pmb{a}})}}$.
On the other hand,
$u(0) = 0$ because 
\[
\adeg (\rest{(P_{\pmb{a}}, q)}{H_1}) = u(0) = 0.
\]
Therefore, $u = 0$ on $\vert z \vert \leq \sqrt{\frac{a_0\vartheta_{\pmb{a}}}{a_1(1-\vartheta_{\pmb{a}})}}$
by the maximal principle.
In a similar way, we can see that $u = 0$ on $\vert z \vert \geq \sqrt{\frac{a_0\theta_{\pmb{a}}}{a_1(1-\theta_{\pmb{a}})}}$.

\medskip
Next we consider the case where $a_0 + a_1 = 1$. 
By Claim~\ref{claim:lem:characteristic:fun:max},
\[
a_1 \log \vert z \vert^2 \leq \log(a_0 + a_1\vert z \vert^2)
\]
on $\PP^1(\CC)$. Thus 
$-a_1\widehat{(z)} \leq \overline{D}_{\pmb{a}}$, and hence the Zariski decomposition of $\overline{D}_{\pmb{a}}$ exists by
\cite[Theorem~9.2.1]{MoArZariski}.
Let $\overline{P}$ be the positive part of $\overline{D}_{\pmb{a}}$. Then $-a_1\widehat{(z)} \leq \overline{P}$.

Let us consider the converse inequality.
Let $t$ be a real number with $t > 1$.
Since $\overline{P} \leq \overline{D}_{\pmb{a}} \leq \overline{D}_{t\pmb{a}}$,
we have $\overline{P} \leq \overline{P}_{t\pmb{a}}$ because $\overline{P}_{t\pmb{a}}$ is the positive part of $\overline{D}_{t\pmb{a}}$
by the previous observation.  Since $\varphi_{t\pmb{a}} = \varphi_{\pmb{a}} + \log (t)$,
we have $\lim_{t\to 1} \vartheta_{t\pmb{a}} = \lim_{t\to 1} \theta_{t\pmb{a}} = a_1$.
Therefore, we can see
\[
\lim_{t\to 1} \overline{P}_{t\pmb{a}} = \overline{P}_{\pmb{a}} = -a_1\widehat{(z)}.
\]
Thus $\overline{P} \leq -a_1\widehat{(z)}$.

\medskip
Finally we consider the case where $a_0+ a_1 < 1$.
Then, by Theorem~\ref{thm:positivity:D:a:b}, $\overline{D}_{\pmb{a}}$ is not pseudo-effective.
Thus the Zariski decomposition does not exist by \cite[Proposition~9.3.2]{MoArZariski}.
\end{proof}

\section{Weak Zariski decomposition of $\overline{D}_{\pmb{a}}$}
\label{sec:counter:example:Zariski:decomposition}
Let $X$ be a $d$-dimensional, projective, generically smooth and normal arithmetic variety.
Let $\overline{D}$ be a big arithmetic $\RR$-divisor of $C^0$-type on $X$.
A decomposition $\overline{D} = \overline{P} + \overline{N}$ is called a {\em weak Zariski decomposition of $\overline{D}$} if the following conditions are satisfied:
\begin{enumerate}
\renewcommand{\labelenumi}{(\arabic{enumi})}
\item
$\overline{P}$ is a nef and big arithmetic $\RR$-divisor of $(C^0 \cap \Tpsh)$-type.

\item
$\overline{N}$ is an effective arithmetic $\RR$-divisor of $C^0$-type.

\item
$\mult_{\Gamma}(N) \leq \mu_{\Gamma}(\overline{D})$ for any horizontal prime divisor $\Gamma$ on $X$, that is,
$\Gamma$ is a reduced and irreducible divisor $\Gamma$ on $X$ such that $\Gamma$ is flat over $\ZZ$.
\end{enumerate}
Note that the Zariski decomposition of a big arithmetic $\RR$-divisor of $C^0$-type on an arithmetic surface is
a weak Zariski decomposition (cf. \cite[Claim~9.3.5.1]{MoArZariski}).
The above property (3) implies that $\mult_{\Gamma}(N) = \mu_{\Gamma}(\overline{D})$ for any horizontal prime divisor $\Gamma$ on $X$.
Indeed, by (2) and (5) in Proposition~\ref{prop:mu:properties},
\[
\mu_{\Gamma}(\overline{D}) \leq \mu_{\Gamma}(\overline{P}) + \mult_{\Gamma}(N) = \mult_{\Gamma}(N) \leq \mu_{\Gamma}(\overline{D}).
\]

\bigskip
From now on,
we use the same notation as in Section~\ref{sec:fund:prop:func}.
Let us begin with the following lemma.

\begin{Lemma}
\label{lem:birat:Zariski:P:n}
Let $f : X \to \PP^n_{\ZZ}$ and $g : Y \to X$ be birational morphisms of projective, generically smooth and normal arithmetic varieties.
If $f^*(\overline{D}_{\pmb{a}})$ admits a weak Zariski decomposition, then $g^*(f^*(\overline{D}_{\pmb{a}}))$ also admits a weak Zariski decomposition.
\end{Lemma}

\begin{proof}
Let $f^*(\overline{D}_{\pmb{a}}) = \overline{P} + \overline{N}$ be a weak Zariski decomposition of $f^*(\overline{D}_{\pmb{a}})$.
We denote birational morphisms $X_{\QQ} \to \PP^n_{\QQ}$ and $Y_{\QQ} \to X_{\QQ}$ by $f_{\QQ}$ and $g_{\QQ}$ respectively.
We set 
\[
\widetilde{\Theta}_{\pmb{a}} = \{ \widetilde{e} \in \RR^{n+1} \mid e \in \Theta_{\pmb{a}} \},
\]
$f^*_{\QQ}(H_i) = \sum_{j} a_{ij} D_j$ for $i=0, \ldots, n$ and $N = \sum_{j} b_j D_j$ on $X_{\QQ}$, where $D_j$'s are reduced and irreducible divisors on $X_{\QQ}$.
Since 
\[
l H_0 + (z^{\pmb{e}}) = (l - \pmb{e}(1) - \cdots - \pmb{e}(n)) H_0 + \pmb{e}(1) H_1 + \cdots + \pmb{e}(n) H_n
\]
for $\pmb{e} \in l \Theta_{\pmb{a}} \cap \ZZ^n$, by Lemma~\ref{lem:generator:mu}, we have
\[
\mu_{D_j}(f^*(\overline{D}_{\pmb{a}})) = \min \left.\left\{ \sum_{i=0}^n x_i a_{ij} \ \right|\ (x_0, \ldots, x_n)  \in \widetilde{\Theta}_{\pmb{a}} \right\}.
\]
Thus
\[
b_j \leq \min \left.\left\{ \sum_{i=0}^n x_i a_{ij} \ \right|\ (x_0, \ldots, x_n)  \in \widetilde{\Theta}_{\pmb{a}} \right\}.
\]
for all $j$.

Here let us see that $g^*(f^*(\overline{D}_{\pmb{a}})) = g^*(\overline{P}) + g^*(\overline{N})$ is a weak Zariski decomposition.
For this purpose, it is sufficient to see that $\mult_{\Gamma}(g^*(N)) \leq \mu_{\Gamma}(g^*(f^*(\overline{D}_{\pmb{a}})))$
for any horizontal prime divisor $\Gamma$ on $Y$.
If we set $c_j = \mult_{\Gamma}(g_{\QQ}^*(D_j))$, then 
\[
d_i := \mult_{\Gamma}(g_{\QQ}^*(f_{\QQ}^*(H_i))) = \sum_{j} a_{ij} c_j.
\]
For $(x_0, \ldots, x_n) \in \widetilde{\Theta}_{\pmb{a}}$,
\[
\sum_i x_i d_i = \sum_j \left( \sum_i x_i a_{ij} \right) c_j \geq \sum_j b_j c_j = \mult_{\Gamma}(g_{\QQ}^*(N)),
\]
which yields $\mu_{\Gamma}(g^*(f^*(\overline{D}_{\pmb{a}}))) \geq \mult_{\Gamma}(g^*(N))$.
\end{proof}

Next let us consider the following lemma:

\begin{Lemma}
\label{lem:convex:set:in:RR2}
Let $\Theta$ be a compact convex set in $\RR^n$ and $p : \RR^n \to \RR^{n-1}$ the projection given by $p(x_1,\ldots,x_n) = (x_1,\ldots,x_{n-1})$.
Then $p(\Theta)$ is a compact convex set in $\RR^{n-1}$ and there exist
a concave function $\theta$ on $p(\Theta)$ and
a convex function $\vartheta$ on $p(\Theta)$ such that
\[
\Theta = \left\{ (x_1, \ldots, x_{n-1}, x_n) \in \RR^n \ \left|\  
\begin{array}{l} (x_1, \ldots, x_{n-1}) \in p(\Theta),\\ 
\vartheta(x_1, \ldots, x_{n-1}) \leq x_n \leq \theta(x_1, \ldots, x_{n-1})
\end{array} \right\}\right. .
\]
\end{Lemma}

\begin{proof}
Obviously $p(\Theta)$ is a compact convex set in $\RR^{n-1}$. For $(x_1,\ldots,x_{n-1}) \in p(\Theta)$, we set
\[
\begin{cases}
\theta(x_1,\ldots,x_{n-1}) := \max \{ x_n \in \RR \mid (x_1,\ldots,x_{n-1},x_n) \in \Theta \},\\
\vartheta(x_1,\ldots,x_{n-1}) := \min \{ x_n \in \RR \mid (x_1,\ldots,x_{n-1},x_n) \in \Theta \}.
\end{cases}
\]
Clearly 
\[
\Theta = \left\{ (x_1, \ldots, x_{n-1}, x_n) \in \RR^n \ \left|\  
\begin{array}{l} (x_1, \ldots, x_{n-1}) \in p(\Theta),\\ 
\vartheta(x_1, \ldots, x_{n-1}) \leq x_n \leq \theta(x_1, \ldots, x_{n-1})
\end{array} \right\}\right. .
\]
We need to show that $\theta$ (resp. $\vartheta$) is a concave (resp. convex) function.
Since 
\[
(x_1, \ldots, x_{n-1}, \theta(x_1, \ldots, x_{n-1})), 
(x'_1, \ldots, x'_{n-1}, \theta(x'_1, \ldots, x'_{n-1})) \in \Theta
\]
for $(x_1, \ldots, x_{n-1}), (x'_1, \ldots, x'_{n-1}) \in p(\Theta)$,
we have 
\[
\lambda (x_1, \ldots, x_{n-1}, \theta(x_1, \ldots, x_{n-1})) + (1-\lambda)(x'_1, \ldots, x'_{n-1}, \theta(x'_1, \ldots, x'_{n-1}))  \in \Theta
\]
for $0 \leq \lambda \leq 1$, which shows
that 
\begin{multline*}
\lambda \theta(x_1, \ldots, x_{n-1}) + (1-\lambda) \theta(x'_1, \ldots, x'_{n-1}) \\
\leq \theta(\lambda (x_1, \ldots, x_{n-1}) + (1-\lambda)(x'_1, \ldots, x'_{n-1})).
\end{multline*}
Thus $\theta$ is concave. Similarly we can see that $\vartheta$ is convex.
\end{proof}

\begin{Remark}
If $p(\Theta)$ is a polytope in Lemma~\ref{lem:convex:set:in:RR2},
then $\theta$ and $\vartheta$ are continuous on $p(\Theta)$ (cf. \cite{GKR}).
In general, $\theta$ and $\vartheta$ are not necessarily continuous on $p(\Theta)$.
Indeed,
let us consider the following set:
\[
\Theta = \{ (x, y, z) \in \RR^3 \mid 0 \leq y \leq 1,\ 0 \leq z \leq 1,\ x^2 \leq yz \}.
\]
Since
\[
x^2 \leq yz\quad\Longleftrightarrow\quad
x^2 + \left( \frac{y - z}{2} \right)^2 \leq \left( \frac{y + z}{2} \right)^2,
\]
we can easily see that $\Theta$ is a compact convex set in $\RR^3$.
Let $p : \RR^3 \to \RR^2$ be the projection given by $p(x,y,z) = (x, y)$.
Then 
\[
p(\Theta) = \{ (x, y) \in \RR^2 \mid x^2 \leq y \leq 1\}.
\]
Moreover, 
$\vartheta$ is given by
\[
\vartheta(x, y) = \begin{cases}
x^2/y & \text{if $(x, y) \not= (0, 0)$}, \\
0 & \text{if $(x,y) = (0,0)$}
\end{cases}
\]
and $\vartheta$ is not continuous at $(0,0)$.
\end{Remark}

\bigskip
Note that $\Theta_{\pmb{a}}$ is a compact convex set of $\RR^n$. 
We say a hyperplane $\alpha_1 x_1 + \cdots + \alpha_n x_n = \beta$ in $\RR^n$
is a {\em supporting hyperplane} of $\Theta_{\pmb{a}}$
at $(b_1, \ldots, b_n) \in \Theta_{\pmb{a}}$
if 
\[
\Theta_{\pmb{a}} \subseteq \{ \alpha_1 x_1 + \cdots + \alpha_n x_n \geq \beta \}\quad\text{and}\quad
\alpha_1 b_1 + \cdots + \alpha_n b_n = \beta.
\]

\begin{Proposition}
\label{prop:theta:tangent}
Let $(b_1, \ldots, b_n) \in \partial(\Theta_{\pmb{a}})$, that is,
$(b_1, \ldots, b_n)$ is a boundary point of $\Theta_{\pmb{a}}$.
We set $b_0 = 1 - b_1 - \cdots - b_n$.
We assume 
\[
a_0 + a_1 + \cdots + a_n > 1\quad\text{and}\quad
\# \{ i \mid 0 \leq i \leq n,\ b_i = 0 \} \leq 1.
\]
Then $\Theta_{\pmb{a}}$ has a unique supporting hyperplane at $(b_1, \ldots, b_n)$.
Moreover, in the case where $b_i = 0$, the supporting hyperplane is given by
\[
\begin{cases}
x_1 + \cdots + x_n = 1 & \text{if $b_0 = 0$}, \\
x_i = 0 & \text{if $b_i = 0$ for some $i$ with $1 \leq i \leq n$}.
\end{cases}
\]
\end{Proposition}

\begin{proof}
Here we set 
\[
\phi_{\pmb{a}}(x_1,\ldots,x_n) = \varphi_{\pmb{a}}(1-x_1-\cdots-x_n, x_1, \ldots, x_n)
\]
on $\Delta_n = \{ (x_1, \ldots, x_n) \in \RR_{\geq 0}^n \mid x_1+ \cdots + x_n \leq 1 \}$.
Then 
\[
\Theta_{\pmb{a}} = \{ (x_1, \ldots, x_n) \in \Delta_n \mid \phi_{\pmb{a}}(x_1, \ldots, x_n) \geq 0 \}.
\]

\medskip
First we assume that $(b_1, \ldots, b_n) \not\in \partial(\Delta_n)$.
Then $\phi_{\pmb{a}}(b_1, \ldots, b_n) = 0$.
Note that, for $(x_1,\ldots, x_n) \in \Delta_{n}  \setminus \partial(\Delta_n)$,
\begin{multline*}
(\phi_{\pmb{a}})_{x_1}(x_1, \ldots, x_n) = \cdots = (\phi_{\pmb{a}})_{x_n}(x_1, \ldots, x_n) = 0\quad
\Longleftrightarrow\quad \\
(x_1, \ldots, x_n) =\left(\frac{a_1}{a_0 + \cdots+ a_n}, \ldots, \frac{a_n}{a_0 + \cdots + a_n}\right),
\end{multline*}
and
$\phi_{\pmb{a}} \left(\frac{a_1}{a_0 + \cdots+ a_n}, \ldots, \frac{a_n}{a_0 + \cdots + a_n}\right) = \log(a_0 + \cdots + a_n) > 0$.
Thus we have 
\[
\left((\phi_{\pmb{a}})_{x_1}(b_1, \ldots, b_n), \ldots, (\phi_{\pmb{a}})_{x_1}(b_1, \ldots, b_n)\right) \not= (0, \ldots, 0),
\]
which means that
$\Theta_{\pmb{a}}$ has a unique supporting hyperplane at $(b_1, \ldots, b_n)$.

\medskip
Next we assume that $(b_1, \ldots, b_n) \in \partial(\Delta_n)$. Considering the following linear transformations:
\[
\left\{ \begin{split}
x'_1 & = x_1, \\
\quad\vdots & \quad \vdots \\
x'_{n-1} & = x_{n-1}, \\
x'_n & = 1 - x_1 - \cdots - x_n,\\
\end{split}\right.
\qquad\qquad
\left\{\begin{split}
x'_1 & = x_1,\\
\quad\vdots & \quad \vdots \\
x'_i & = x_n, \\
\quad\vdots & \quad \vdots \\
x'_n & = x_i,
\end{split}\right.
\]
we may assume $b_n = 0$.
Note that $(b_1, \ldots, b_{n-1}) \in \Delta_{n-1} \setminus \partial(\Delta_{n-1})$.
Let $p : \RR^n \to \RR^{n-1}$ be the projection given by $p(x_1, \ldots, x_n) = (x_1, \ldots, x_{n-1})$.
By Lemma~\ref{lem:convex:set:in:RR2},
there are 
a concave function $\theta$ on $p(\Theta_{\pmb{a}})$ and
a convex function $\vartheta$ on $p(\Theta_{\pmb{a}})$ such that
\[
\Theta_{\pmb{a}} = \left\{ (x_1, \ldots, x_{n-1}, x_n) \ \left|\ \begin{array}{l}
(x_1, \ldots, x_{n-1}) \in p(\Theta_{\pmb{a}}),\\
\vartheta(x_1, \ldots, x_{n-1}) \leq x_n \leq \theta(x_1, \ldots, x_{n-1}) \end{array} \right\}\right..
\]

\begin{Claim}
\label{claim:prop:theta:tangent:1}
$(b_1, \ldots, b_{n-1})$ is an interior point of $p(\Theta_{\pmb{a}})$.
In particular, $\vartheta$ is continuous around $(b_1, \ldots, b_{n-1})$ \rom{(}cf. \cite[Theorem~2.2]{Gru}\rom{)}.
\end{Claim}

\begin{proof}
Let us consider a function $\psi : [0, 1 - b_1 - \cdots - b_{n-1}] \to \RR$ given by
$\psi(t) = \phi_{\pmb{a}}(b_1, \ldots, b_{n-1}, t)$.
Note that 
\[
\psi'(t) = \log\frac{a_n}{a_0}\left(\frac{1 - b_1 - \cdots - b_{n-1}}{t} - 1\right).
\]
Thus
\[
\phi_{\pmb{a}}\left(b_1, \ldots, b_{n-1}, \frac{a_n(1-b_1-\cdots-b_{n-1})}{a_0 + a_n}\right) > \phi_{\pmb{a}}(b_1, \ldots, b_{n-1}, 0) \geq 0.
\]
Therefore, as $\left(b_1, \ldots, b_{n-1}, \frac{a_n(1-b_1-\cdots-b_{n-1})}{a_0 + a_n}\right) \in \Delta_n \setminus \partial(\Delta_n)$,
we can find a sufficiently small positive number $\epsilon$ such that
\[
\prod_{i=1}^{n-1}
(b_i - \epsilon, b_i + \epsilon)  \times
\left( \frac{a_n(1-b_1-\cdots-b_{n-1})}{a_0 + a_n} - \epsilon, \frac{a_n(1-b_1-\cdots-b_{n-1})}{a_0 + a_n}+ \epsilon \right)
\]
is a subset of $\Theta_{\pmb{a}}$, and hence
\[
(b_1, \ldots, b_{n-1}) \in  \prod_{i=1}^{n-1}
(b_i - \epsilon, b_i + \epsilon) \subseteq p(\Theta_{\pmb{a}}).
\]
\end{proof}

We set $\pmb{a}' = (a_0, \ldots, a_{n-1})$. Then
\[
\Theta_{\pmb{a}'} = \{ (x_1, \ldots, x_{n-1}) \in \RR^{n-1} \mid (x_1, \ldots, x_{n-1}, 0) \in \Theta_{\pmb{a}} \}.
\]
Clearly $(b_1, \ldots, b_{n-1}) \in \Theta_{\pmb{a}'}$ and $\vartheta \equiv 0$ on $\Theta_{\pmb{a}'}$.

\begin{Claim}
\label{claim:prop:theta:tangent:2}
$\vartheta$ is a continuously differentiable function around $(b_1, \ldots, b_{n-1})$ such that
\[
\vartheta_{x_1}(b_1, \ldots, b_{n-1}) = \cdots = \vartheta_{x_{n-1}}(b_1, \ldots, b_{n-1}) = 0.
\]
\end{Claim}

\begin{proof}
By Claim~\ref{claim:prop:theta:tangent:1},
there is a positive number $\epsilon$ such that \[
b_1 - \epsilon > 0, \ldots, b_{n-1} - \epsilon > 0, \ 
(b_1 + \epsilon) + \cdots + (b_{n-1} + \epsilon) < 1
\]
and
$\vartheta$ is continuous on $U = \prod_{i=1}^{n-1} (b_i - \epsilon, b_i + \epsilon)$.
If $(x_1, \ldots, x_{n-1}) \in U \setminus \Theta_{\pmb{a}'}$, then
$\vartheta(x_1, \ldots, x_{n-1}) > 0$, and hence 
\[
\phi_{\pmb{a}}(x_1, \ldots, x_{n-1}, \vartheta(x_1, \ldots, x_{n-1})) = 0
\]
for $(x_1, \ldots, x_{n-1}) \in U \setminus \Theta_{\pmb{a}'}$. Note that
\addtocounter{Claim}{1}
\begin{equation}
\label{eqn:prop:theta:tangent:1}
(\phi_{\pmb{a}})_{x_i} = \log \frac{a_i}{a_0} \left( \frac{1 - x_1 - \cdots - x_n}{x_i} \right).
\end{equation}
Since $\vartheta(b_1, \ldots, b_{n-1}) = 0$, choosing a smaller $\epsilon$ if necessarily, we may assume that
\[
(\phi_{\pmb{a}})_{x_n}(x_1, \ldots, x_{n-1}, \vartheta(x_1, \ldots, x_{n-1})) > 0
\]
for all $(x_1, \ldots, x_{n-1})\in U \setminus \Theta_{\pmb{a}'}$. Thus, by using the implicit function theorem,
$\vartheta$ is a $C^{\infty}$ function on $U \setminus \Theta_{\pmb{a}'}$ and
\addtocounter{Claim}{1}
\begin{equation}
\label{eqn:prop:theta:tangent:2}
\vartheta_{x_i}(x_1, \ldots, x_{n-1}) = -
\frac{(\phi_{\pmb{a}})_{x_i}(x_1, \ldots, x_{n-1}, \vartheta(x_1, \ldots, x_{n-1}))}%
{(\phi_{\pmb{a}})_{x_n}(x_1, \ldots, x_{n-1}, \vartheta(x_1, \ldots, x_{n-1}))}.
\end{equation}
Let us consider a function $\gamma_i$ on $U$ given by
\[
\gamma_i(x_1, \ldots, x_{n-1}) = \begin{cases}
0 & \text{if $(x_1, \ldots, x_{n-1}) \in U \cap \Theta_{\pmb{a}'}$},\\
\vartheta_{x_i}(x_1, \ldots, x_{n-1}) & \text{if $(x_1, \ldots, x_{n-1}) \in U \setminus \Theta_{\pmb{a}'}$}.
\end{cases}
\]
Then, by using \eqref{eqn:prop:theta:tangent:1} and \eqref{eqn:prop:theta:tangent:2}, 
it is easy to see that $\gamma_i$ is continuous on $U$.
Thus the claim follows.
\end{proof}

The above claim shows that $\Theta_{\pmb{a}}$ has the unique supporting hyperplane at $(b_1, \ldots, b_n)$ and
it is given by $x_n = 0$.
\end{proof}

\begin{Corollary}
\label{cor:tangent:point:interior}
We assume that $a_0 < 1$ and $a_0 + a_1 + \cdots + a_n \geq 1$.
Let $\alpha_1, \ldots, \alpha_n \in \RR_{>0}$ and $(b_1, \ldots, b_n) \in \Theta_{\pmb{a}}$ such that
\[
\alpha_1 b_1 + \cdots + \alpha_n b_n  = \min \{ \alpha_1 x_1 + \cdots + \alpha_n x_n \mid (x_1, \ldots, x_n) \in \Theta_{\pmb{a}} \}.
\]
Then $(b_1, \ldots, b_n) \not\in \partial(\Delta_n)$.
\end{Corollary}

\begin{proof}
We prove it by induction on $n$.
If $n=1$, then the assertion is obvious, so that we may assume $n > 1$.
If $a_0 + \cdots + a_n = 1$, then
\[
\Theta_{\pmb{a}} = \left\{ \left(\frac{a_1}{a_0 + \cdots + a_n}, \ldots, \frac{a_n}{a_0 + \cdots + a_n} \right) \right\}.
\]
In this case, the assertion is also obvious. Thus we may assume that $a_0 + \cdots + a_n > 1$.

We assume that $b_i = 0$ for some $1 \leq i \leq n$.
Then, since $\Theta_{\pmb{a}} \cap \{ x_i = 0 \} \not= \emptyset$, we have
\[
a_1 + \cdots + a_{i-1} + a_{i+1} + \cdots + a_n \geq 1.
\]
Thus, by the hypothesis of induction, 
\[
b_1 \not= 0, \ldots, b_{i-1} \not= 0, b_{i+1} \not= 0, \ldots, b_n \not= 0, b_1 + \cdots + b_n \not= 1.
\]
Therefore, by Proposition~\ref{prop:theta:tangent},
we have
the unique supporting hyperplane $x_i = 0$ of $\Theta_{\pmb{a}}$ at $(b_1, \ldots, b_n)$. 
On the other hand, $\alpha_1 x_1 + \cdots + \alpha_n x_n = \alpha_1 b_1 + \cdots + \alpha_n b_n$ is
also a supporting hyperplane of $\Theta_{\pmb{a}}$ at $(b_1, \ldots, b_n)$.
This is a contradiction. 

Next we assume that $b_1 + \cdots + b_n = 1$.
Since $b_i \not= 0$ for all $i$, by Proposition~\ref{prop:theta:tangent},
the unique supporting hyperplane of $\Theta_{\pmb{a}}$ at $(b_1, \ldots, b_n)$
is $x_1 + \cdots + x_n = 1$, which yields $\alpha_1 = \cdots = \alpha_n$, and hence
$\Theta_{\pmb{a}} \subseteq \{ x_1 + \cdots + x_n = 1 \}$.
This is a contradiction because
\[
\left(\frac{a_1}{a_0 + \cdots + a_n}, \ldots, \frac{a_n}{a_0 + \cdots + a_n} \right) \in \Theta_{\pmb{a}},
\]
as required.
\end{proof}

\begin{Theorem}
We assume that $n \geq 2$ and $\overline{D}_{\pmb{a}}$ is big.
Then $\overline{D}_{\pmb{a}}$ is nef if and only if 
there is a birational morphism $f : X \to \PP^n_{\ZZ}$ of projective, generically smooth and normal arithmetic varieties such that
$f^*(\overline{D}_{\pmb{a}})$ admits  a weak Zariski decomposition on $X$.
\end{Theorem}

\begin{proof}
If $\overline{D}_{\pmb{a}}$ is nef, then $\overline{D}_{\pmb{a}} = \overline{D}_{\pmb{a}} + (0,0)$ is a
weak Zariski decomposition. Next we assume that
$\overline{D}_{\pmb{a}}$ is not nef and there is a birational morphism $f : X \to \PP^n_{\ZZ}$ of projective, generically smooth and normal arithmetic varieties such that
$f^*(\overline{D}_{\pmb{a}})$ admits  a weak Zariski decomposition $f^*(\overline{D}_{\pmb{a}}) = \overline{P} + \overline{N}$ on $X$.
By our assumptions, $a_0 + \cdots + a_n > 1$ and $a_i < 1$ for some $i$.
Renumbering the homogeneous coordinate $T_0, \ldots, T_n$, we may assume $a_0 < 1$.
Let $\xi$ be the generic point of $H_1 \cap \cdots \cap H_n$, that is, $\xi = (1:0:\cdots:0) \in \PP^n(\QQ)$.
Let $L_i$  be the strict transform of $H_i$ by $f$ for $i=0, \ldots, n$.
We denote the birational morphism $X_{\QQ} \to \PP^n_{\QQ}$ by $f_{\QQ}$.
Let $f' : X' \to \PP^n_{\ZZ}$ be the blowing-up along $H_1 \cap \cdots \cap H_n$.
By using Lemma~\ref{lem:birat:Zariski:P:n} and \cite{Hiro}, we may assume the following:
\begin{enumerate}
\renewcommand{\labelenumi}{(\arabic{enumi})}
\item
Let $\Sigma$ be the
exceptional set of $f_{\QQ} : X_{\QQ} \to \PP^n_{\QQ}$.
Then $\Sigma$ is a divisor on $X_{\QQ}$ and $(\Sigma + (L_0)_{\QQ} + \cdots + (L_n)_{\QQ})_{\operatorname{red}}$ is a normal crossing divisor on $X_{\QQ}$.

\item
There is a birational morphism $g : X \to X'$ such that the following diagram is commutative:
\[
\xymatrix{
X \ar[dr]^{g} \ar[dd]_{f} & \\
 & X' \ar[dl]^{f'} \\
\PP^n_{\ZZ} &
}
\]
\end{enumerate}

\begin{Claim}
There are $\xi' \in X(\QQ)$ and a reduced and irreducible divisor $E$ on $X_{\QQ}$ with the following properties:
\begin{enumerate}
\renewcommand{\labelenumi}{(\alph{enumi})}
\item
$f_{\QQ}(\xi') = \xi$ and $\xi' \in E \cap (L_n)_{\QQ}$.

\item
$E$ and $(L_n)_{\QQ}$ is non-singular at $\xi'$.

\item
$E$ is exceptional with respect to $f_{\QQ} : X_{\QQ} \to \PP^n_{\QQ}$.

\item
There are positive integers $\alpha_1, \ldots, \alpha_n$ such that
\begin{align*}
\hspace{3em} f_{\QQ}^*(H_i) & = \alpha_i E + (\text{the sum of divisors which do not pass through $\xi'$}) \\
\intertext{for $i=1, \ldots, n-1$ and}
f_{\QQ}^*(H_n) & = (L_n)_{\QQ} + \alpha_n E + (\text{the sum of divisors which do not pass through $\xi'$}).
\end{align*}
\end{enumerate}
\end{Claim}

\begin{proof}
Let $L'_n$ be the strict transform of $H_n$ by $f'$ and $\Sigma'$ the exceptional set of $f'_{\QQ} : X'_{\QQ} \to \PP^n_{\QQ}$.
Then $\Sigma' = \PP^{n-1}_{\QQ}$ and $D' := (L'_n)_{\QQ} \cap \Sigma' = \PP^{n-2}_{\QQ}$.
Let $h : L_n \to L'_n$ and $h_{\QQ} : (L_n)_{\QQ} \to (L_n')_{\QQ}$ be the birational morphisms induced by $g : X \to X'$ and
$g_{\QQ} : X_{\QQ} \to X'_{\QQ}$ respectively.
Let $D$ be the strict transformation of $D'$ by $h_{\QQ}$.
As before, let $\Sigma$ be the exceptional set of $f_{\QQ} : X_{\QQ} \to \PP^n_{\QQ}$.
Let
\[
(\Sigma + (L_0)_{\QQ} + \cdots + (L_n)_{\QQ})_{\operatorname{red}} = (L_0)_{\QQ} + \cdots + (L_n)_{\QQ} + E_0 + \cdots + E_l
\]
be the irreducible decomposition such that $E_i$'s are exceptional with respect to $f_{\QQ}$.
Since $D \subseteq (L_n)_{\QQ} \cap \Sigma$,
there is $E_i$ such that $D \subseteq (L_n)_{\QQ} \cap E_i$. 
Renumbering $E_0, \ldots, E_l$, we may assume that $E_i = E_l$.
As $(L_0)_{\QQ} + \cdots + (L_n)_{\QQ} + E_0 + \cdots + E_l$ is a normal crossing divisor on $X_{\QQ}$,
we have
\[
\begin{cases}
D \cap \Sing((L_n)_{\QQ}) \subsetneq D,\  D \cap \Sing(E)  \subsetneq D,\\
D \cap (L_i)_{\QQ} \subsetneq D\ (i=0, \ldots, n-1), \\ 
D \cap E_j \subsetneq D\ (j=0, \ldots, l-1).
\end{cases}
\]
Note that $D(\QQ)$ is dense in $D$ because $D \to D'$ is birational.
Thus we can find $\xi' \in D(\QQ)$ such that
\[
\xi' \not\in (D \cap \Sing((L_n)_{\QQ})) \cup (D \cap \Sing(E)) \cup \bigcup_{i=0}^{n-1} (D \cap (L_i)_{\QQ}) \cup \bigcup_{j=0}^{l-1} (D \cap E_j).
\]
Therefore the claim follows.
\end{proof}

Note that 
\begin{multline*}
f_{\QQ}^*(l H_0 +(z_1^{e_1}\cdots z_n^{e_n})) = f_{\QQ}^*((l-e_1-\cdots-e_n)  H_0 + e_1 H_1 + \cdots + e_n H_n) \\
= e_n (L_n)_{\QQ} + (\alpha_1 e_1 + \cdots + \alpha_n e_n)E \\
+ (\text{the sum of divisors which
do not pass through $\xi'$}).
\end{multline*}
Therefore, by Lemma~\ref{lem:generator:mu},
\[
\begin{cases}
\mu_{\xi'}(f^*(\overline{D}_{\pmb{a}})) = \min \{ \alpha_1 x_1  + \cdots + \alpha_{n-1} x_{n-1} + (\alpha_n+1)x_n \mid (x_1, \cdots, x_n) \in \Theta_{\pmb{a}} \}, \\
\mu_E(f^*(\overline{D}_{\pmb{a}}))  = \min \{ \alpha_1 x_1 + \cdots + \alpha_n x_n  \mid (x_1, \cdots, x_n) \in \Theta_{\pmb{a}} \}, \\
\mu_{L_n}(f^*(\overline{D}_{\pmb{a}}))  = \min \{ x_n  \mid (x_1, \ldots, x_n) \in \Theta_{\pmb{a}} \}
\end{cases}
\]
Further, 
\[
\mult_{\xi'}(N) = \mult_{E}(N) + \mult_{L_n}(N) \leq \mu_E(f^*(\overline{D}_{\pmb{a}}))  + \mu_{L_n}(f^*(\overline{D}_{\pmb{a}})).
\]
By (2) and (5) in Proposition~\ref{prop:mu:properties},
\[
0 = \mu_{\xi'}(\overline{P}) \geq \mu_{\xi'}(f^*(\overline{D}_{\pmb{a}})) - \mult_{\xi'}(N).
\]
Therefore, if we set
\[
\begin{cases}
A = \min \{ \alpha_1 x_1  + \cdots + \alpha_{n-1} x_{n-1} + (\alpha_n+1)x_n \mid (x_1, \cdots, x_n) \in \Theta_{\pmb{a}} \},\\
B = \min \{ \alpha_1 x_1 + \cdots + \alpha_n x_n  \mid (x_1, \cdots, x_n) \in \Theta_{\pmb{a}} \},\\
C =  \min \{ x_n  \mid (x_1, \ldots, x_n) \in \Theta_{\pmb{a}} \},
\end{cases}
\]
then we have $0 \geq A - B - C$.
We choose $(b_1, \ldots, b_n) \in \Theta_{\pmb{a}}$ such that 
\[
A = \alpha_1 b_1  + \cdots + \alpha_{n-1} b_{n-1} + (\alpha_n+1)b_n.
\] 
Thus, as $\alpha_1 b_1 + \cdots + \alpha_n b_n \geq B$  and $b_n \geq C$, we have
\begin{multline*}
0 \geq A - B - C \\
\geq  \alpha_1 b_1  + \cdots + \alpha_{n-1} b_{n-1} + (\alpha_n+1)b_n - (\alpha_1 b_1 + \cdots + \alpha_n b_n) - b_n = 0,
\end{multline*}
which implies $\alpha_1 b_1 + \cdots + \alpha_n b_n = B$  and $b_n = C$.
On the other hand, by Corollary~\ref{cor:tangent:point:interior}, $(b_1, \ldots, b_n) \not\in \partial(\Delta_n)$, and hence
there is a unique supporting hyperplane of $\Theta_{\pmb{a}}$ at $(b_1, \ldots, b_n)$ by Proposition~\ref{prop:theta:tangent}.
This is a contradiction because 
\[
\begin{cases}
\alpha_1 x_1  + \cdots + \alpha_{n-1} x_{n-1} + (\alpha_n+1)x_n = A, \\
\alpha_1 x_1  + \cdots + \alpha_{n-1} x_{n-1} + \alpha_nx_n = B,\\
x_n = C
\end{cases}
\]
are distinct supporting hyperplanes of $\Theta_{\pmb{a}}$ at $(b_1, \ldots, b_n)$.
\end{proof}

\section{Fujita's approximation of $\overline{D}_{\pmb{a}}$}
\label{sec:Fujita:app}
Fujita's approximation of arithmetic divisors has established by Chen and Yuan (cf. \cite{HChenFujita},
\cite{YuanVol}, \cite{MoArLin} and \cite{MoArZariski}).
In this section, we consider Fujita's approximation of $\overline{D}_{\pmb{a}}$ in terms of rational
interior points of $\Theta_{\pmb{a}}$.

First of all, we fix notation.
Let $\pmb{x}_1, \ldots, \pmb{x}_r  \in \RR^n$ and $\phi_1, \ldots, \phi_r \in \RR$.
We define a function $\phi_{(\pmb{x}_1, \phi_1), \ldots, (\pmb{x}_r, \phi_r)}$ on $\Theta = \Conv \{ \pmb{x}_1, \ldots, \pmb{x}_r \}$ to be
\[
\phi_{(\pmb{x}_1, \phi_1), \ldots, (\pmb{x}_r, \phi_r)}(\pmb{x}) := \max \left\{ \sum_{i=1}^r \lambda_i \phi_i \ \left| \ 
 \begin{array}{l}  \pmb{x} = \sum_{i=1}^r \lambda_i \pmb{x}_i,  \\
 \lambda_1,\ldots, \lambda_r \in \RR_{\geq 0},\ \sum_{i=1}^r \lambda_i = 1 \\
\end{array}\right\}\right..
\]
Note that 
\[
\phi_{(\pmb{x}_1, \phi_1), \ldots, (\pmb{x}_r, \phi_r)}(\pmb{x}) = \max \{ \phi \in \RR \mid (\pmb{x}, \phi) \in \Conv \{ (\pmb{x}_1, \phi_1), \ldots, (\pmb{x}_r, \phi_r) \} \subseteq   \RR^n \times \RR \}.
\]
Thus we can easily see that $\phi_{(\pmb{x}_1, \phi_1), \ldots, (\pmb{x}_r, \phi_r)}$ is a continuous 
function on $\Theta$ (cf. \cite{GKR}).

Let $\varphi$ be a continuous concave function on $\Theta$.
Then $\phi_{(\pmb{x}_1, \varphi(\pmb{x}_1)), \ldots, (\pmb{x}_r, \varphi(\pmb{x}_r))} \leq \varphi$.
Moreover, for a positive number $\epsilon$, if we add sufficiently many points
$\pmb{x}_{r+1}, \ldots, \pmb{x}_m  \in \Theta$ to $\{ \pmb{x}_1, \ldots, \pmb{x}_r \}$, then
\[
\varphi - \epsilon \leq \phi_{(\pmb{x}_1, \varphi(\pmb{x}_1)), \ldots, (\pmb{x}_r, \varphi(\pmb{x}_r)), (\pmb{x}_{r+1}, \varphi(\pmb{x}_{r+1})), \ldots,(\pmb{x}_m, \varphi(\pmb{x}_m))} \leq \varphi.
\]

\bigskip
From now on,
we use the same notation as in Section~\ref{sec:fund:prop:func}.
We assume that $\overline{D}_{\pmb{a}}$ is big.

\renewcommand{\theClaim}{\arabic{section}.\arabic{Theorem}}
\addtocounter{Theorem}{1}
\begin{Claim}
\label{claim:sec:Fujita:app:1}
For a given positive number $\epsilon$, we can find rational interior points
$\pmb{x}_1, \ldots, \pmb{x}_r$ of $\Theta_{\pmb{a}}$, that is,
$\pmb{x}_1, \ldots, \pmb{x}_r \in \operatorname{Int}(\Theta_{\pmb{a}}) \cap \QQ^n$ such that
\[
\frac{(n+1)!}{2} \int_{\Theta} \phi_{(\pmb{x}_1, \varphi_{\pmb{a}}(\widetilde{\pmb{x}}_1)), \ldots, (\pmb{x}_r, \varphi_{\pmb{a}}(\widetilde{\pmb{x}}_r))}(\pmb{x}) d\pmb{x} > \avol(\overline{D}_{\pmb{a}}) - \epsilon,
\]
where $\Theta = \Conv \{ \pmb{x}_1, \ldots, \pmb{x}_r \}$.
\end{Claim}

\begin{proof}
First of all, we can find $\pmb{x}_1, \ldots, \pmb{x}_{r'} \in \operatorname{Int}(\Theta_{\pmb{a}}) \cap \QQ^n$ such that
\[
\frac{(n+1)!}{2} \int_{\Theta} \varphi_{\pmb{a}}(\widetilde{\pmb{x}}) d\pmb{x} > \avol(\overline{D}_{\pmb{a}}) - \epsilon,
\]
where $\Theta = \Conv \{ \pmb{x}_1, \ldots, \pmb{x}_{r'} \}$. Thus, adding more points $\pmb{x}_{r'+1}, \ldots, \pmb{x}_{r} \in \Theta \cap \QQ^n$ to
$\{ \pmb{x}_1, \ldots, \pmb{x}_{r'} \}$, we have
\[
\frac{(n+1)!}{2} \int_{\Theta} \phi_{(\pmb{x}_1, \varphi_{\pmb{a}}(\widetilde{\pmb{x}}_1)), \ldots, (\pmb{x}_r, \varphi_{\pmb{a}}(\widetilde{\pmb{x}}_r))}(\pmb{x}) d\pmb{x} > \avol(\overline{D}_{\pmb{a}}) - \epsilon.
\]
\end{proof}

We choose a sufficiently small positive number $\delta$ such that
\begin{enumerate}
\renewcommand{\labelenumi}{(\alph{enumi})}
\item
$\Theta \subseteq \Theta_{e^{-\delta}\pmb{a}}$ and

\item
${\displaystyle
\frac{(n+1)!}{2} \int_{\Theta} \phi_{(\pmb{x}_1, \varphi_{e^{-\delta}\pmb{a}}(\widetilde{\pmb{x}}_1)), \ldots, (\pmb{x}_r, \varphi_{e^{-\delta}\pmb{a}}(\widetilde{\pmb{x}}_r))}(\pmb{x}) d\pmb{x} > \avol(\overline{D}_{\pmb{a}}) - \epsilon}$.
\end{enumerate}
We set $\pmb{a}' = e^{-\delta} \pmb{a}$.
By virtue of \cite[Theorem3.2.3]{MoArZariski}, we can find positive integer $l_0$ such that
\begin{enumerate}
\renewcommand{\labelenumi}{(\alph{enumi})}
\setcounter{enumi}{2}
\item
$\log \dist(H^0(lH_0) \otimes \CC; l_0 g_{\pmb{a}'}) \leq l_0 \delta$ and

\item
$l_0 \pmb{x}_1, \ldots, l_0 \pmb{x}_r \in \ZZ_{\geq 0}^n$.
\end{enumerate}
Let us consider the following $\ZZ$-module:
\[
V := \bigoplus_{i=1}^r \ZZ z^{l_0 \pmb{x}_i} \subseteq H^0(\PP^n_{\ZZ}, l_0H_0).
\]
Then we have a birational morphisms $\mu : Y \to \PP^n_{\ZZ}$ of projective, generically smooth and normal arithmetic varieties
such that the image of
\[
V \otimes_{\ZZ} \OO_Y \to \OO_Y(\mu^*(l_0H_0)) 
\]
is invertible, that is,
there is an effective Cartier divisor $F$ on $Y$ such that 
\[
V \otimes_{\ZZ} \OO_Y \to \OO_Y(\mu^*(l_0H_0)-F) 
\]
is surjective. Here we set
\[
\begin{cases}
Q := \mu^*(l_0 H_0) - F, \\
g_F := \mu^*\left( -\log \dist(V \otimes \CC; l_0 g_{\pmb{a}'}) + l_0 \delta\right), \\
g_Q := \mu^*\left(l_0 g_{\pmb{a}'} + \log \dist(V \otimes \CC; l_0 g_{\pmb{a}'})\right).
\end{cases}
\]

\addtocounter{Theorem}{1}
\begin{Claim}
\label{claim:sec:Fujita:app:2}
\begin{enumerate}
\renewcommand{\labelenumi}{(\roman{enumi})}
\item
$g_Q + g_F = \mu^*(l_0g_{\pmb{a}})$.

\item
$g_Q$ is a $Q$-Green function of $(C^{\infty} \cap \Tpsh)$-type and
$\overline{Q} := (Q, g_Q)$ is nef.

\item
$g_F$ is an $F$-Green function of $C^{\infty}$-type and  $g_F \geq 0$.

\item
If we set $\overline{P} = (P, g_P) = (1/l_0)\overline{Q}$, then, for $\pmb{e} \in l \Theta \cap \ZZ^n$,
$\mu^*(z^{\pmb{e}}) \in H^0(lP)$ and
\[
\vert \mu^*(z^{\pmb{e}}) \vert^2_{lg_P} \leq \exp\left(-l \phi_{(\pmb{x}_1, \varphi_{\pmb{a}'}(\widetilde{\pmb{x}}_1)), \ldots, (\pmb{x}_r, \varphi_{\pmb{a}'}(\widetilde{\pmb{x}}_r))}(\pmb{e}/l)\right).
\]
\end{enumerate}
\end{Claim}

\begin{proof}
(i) is obvious. (ii) is a consequence of Lemma~\ref{lem:nef:div} below.
The first assertion of (iii) follows from (i) and (ii), and
the second
follows from (c). 

(iv) Let us consider arbitrary $\lambda_1, \ldots, \lambda_r \in \RR$ such that
$\pmb{e}/l = \lambda_1 \pmb{x}_1 + \cdots + \lambda_r \pmb{x}_r$ and
$\lambda_1 + \cdots + \lambda_r = 1$. Then, since $Q+ (\mu^*(z^{l_0 \pmb{x}_i}))  \geq 0$ for all $i$,
\begin{align*}
lP + (\mu^*(z^{\pmb{e}})) & = (l/l_0) Q + \sum_{i=1}^r \lambda_i (l/l_0) (\mu^*(z^{l_0 \pmb{x}_i})) \\
& = 
\sum_{i=1}^r \lambda_i (l/l_0) \left( Q + (\mu^*(z^{l_0 \pmb{x}_i})) \right) \geq 0,
\end{align*}
and hence $\mu^*(z^{\pmb{e}}) \in H^0(lP)$.
Moreover,  by using \cite[Proposition~3.2.1]{MoArZariski} and Proposition~\ref{prop:cal:inner:product},
\begin{align*}
\vert \mu^*(z^{\pmb{e}}) \vert^2_{lg_P} & = \vert \mu^*(z^{\pmb{e}}) \vert^2 \exp(-(l/l_0)g_Q) \\
& = \prod_{i=1}^r 
\left( \vert \mu^*(z^{l_0 \pmb{x}_i}) \vert^2\right)^{\lambda_i(l/l_0)} 
\frac{\exp(-l \mu^* (g_{\pmb{a}'}))}{\mu^* (\dist(V \otimes \CC; l_0 g_{\pmb{a}'}))^{l/l_0}} \\
& = \prod_{i=1}^r \mu^* \left( \frac{\vert z^{l_0 \pmb{x}_i} \vert^2_{l_0g_{\pmb{a}'}}}{\dist(V \otimes \CC; l_0 g_{\pmb{a}'})} \right)^{\lambda_i(l/l_0)} \leq  \prod_{i=1}^r \left( \Vert z^{l_0 \pmb{x}_i} \Vert^2_{l_0g_{\pmb{a}'}} \right)^{\lambda_i(l/l_0)} \\
& =
 \prod_{i=1}^r \exp(-l_0 \varphi_{\pmb{a}'}(\widetilde{\pmb{x}}_i) )^{\lambda_i(l/l_0)} = \exp\left( -l \sum_{i=1}^r \lambda_i \varphi_{\pmb{a}'}(\widetilde{\pmb{x}}_i) \right).
\end{align*}
Thus (iv) follows.
\end{proof}

\begin{Lemma}
\label{lem:nef:div}
Let $\mu : Y \to X$ be a birational morphism of projective, generically smooth and normal arithmetic varieties.
Let $\overline{D}$ be an arithmetic $\RR$-divisor of $C^0$-type on $X$ and $S$ a subset of $\aH(X, \overline{D})$.
We assume that there is an effective $\RR$-divisor $E$ on $Y$ with the following properties:
\begin{enumerate}
\renewcommand{\labelenumi}{(\arabic{enumi})}
\item $\mu^*(D) - E \in \Div(Y)$, that is, $\mu^*(D) - E$ is a Cartier divisor.

\item $\mu^*(s) \in H^0(Y, \mu^*(D) - E)$ for all $s \in S$ and
\[
\bigcap_{s \in S} \Supp(\mu^*(D) - E + (\mu^*(s))) = \emptyset.
\]

\end{enumerate}
We set 
\[
M := \mu^*(D) - E\quad\text{and}\quad
g_M := \mu^*(g + \log \dist(\langle S \rangle_{\CC}; g)).
\] 
Then $g_M$ is an $M$-Green function of $(C^{\infty} \cap \Tpsh)$-type and
$(M, g_M)$ is nef.
\end{Lemma}

\begin{proof}
Let $e_1, \ldots, e_N$ be an orthonormal basis of $\langle S \rangle_{\CC}$ with respect to
$\langle\ ,\ \rangle_{g}$.
We fix $y \in Y(\CC)$.
Let $f$ be a local equation of $\mu^*(D) - E$ around $y$.
We set $s_j = \mu^*(e_j)f$ for $j=1, \ldots, N$.
Then $s_1, \ldots, s_N$ are holomorphic around $y$ and $s_j(y) \not= 0$ for some $j$.
On the other hand,
\[
g_M = \log \left( \sum_{j=1}^N | \mu^*(e_j) |^2 \right) = 
-\log | f |^2 + \log \left( \sum_{j=1}^N \left| s_j \right|^2 \right)
\]
around $y$. Thus $g_M$ is an $M$-Green function of $(C^{\infty} \cap \Tpsh)$-type.
By virtue of \cite[Proposition~3.1]{MoArZariski}, we have
\[
| s |^2_{g} \leq \langle s, s \rangle_{g} \dist(\langle S \rangle_{\CC}; g) \leq \dist(\langle S \rangle_{\CC}; g),
\]
which yields $\mu^*(s) \in \aH(Y, \overline{M})$ for all $s \in S$.
Let $C$ be a $1$-dimensional closed integral subscheme on $Y$.
Then there is $s \in S$ such that $C \not\subseteq \Supp(M + (\mu^*(s)))$. Thus $\adeg(\rest{(M, g_M)}{C}) \geq 0$.
\end{proof}

\bigskip
Finally let us see that $\avol(\overline{P}) > \avol(\overline{D}_{\pmb{a}}) - \epsilon$.
We fix an $F_{\infty}$-invariant  volume form $\Phi$ on $Y$ with $\int_{Y(\CC)} \Phi = 1$.
Using $\Phi$ and $lg_P$, we can give the inner product $\langle\ ,\ \rangle_{lg_P}$ on $H^0(lP)$.
Then, by (iv) in the above claim,
\[
\langle \mu^*(z^{\pmb{e}}),  \mu^*(z^{\pmb{e}}) \rangle_{lg_P} \leq \exp\left(-l \phi_{(\pmb{x}_1, \varphi_{\pmb{a}'}(\widetilde{\pmb{x}}_1)), \ldots, (\pmb{x}_r, \varphi_{\pmb{a}'}(\widetilde{\pmb{x}}_r))}(\pmb{e}/l)\right).
\]
Here we consider positive definite symmetric real matrices $A_l = (a_{\pmb{e}, \pmb{e}'})_{\pmb{e}, \pmb{e}' \in l\Theta \cap \ZZ^n}$ and
$A'_l = (a'_{\pmb{e}, \pmb{e}'})_{\pmb{e}, \pmb{e}' \in l\Theta \cap \ZZ^n}$ given by
\begin{align*}
a_{\pmb{e}, \pmb{e}'} & = \langle \mu^*(z^{\pmb{e}}),  \mu^*(z^{\pmb{e}'}) \rangle_{lg_P} \\
\intertext{and}
a'_{\pmb{e}, \pmb{e}'} & = \begin{cases}
\exp\left(-l \phi_{(\pmb{x}_1, \varphi_{\pmb{a}'}(\widetilde{\pmb{x}}_1)), \ldots, (\pmb{x}_r, \varphi_{\pmb{a}'}(\widetilde{\pmb{x}}_r))}(\pmb{e}/l)\right) & \text{if $\pmb{e} = \pmb{e}'$}, \\
\langle \mu^*(z^{\pmb{e}}),  \mu^*(z^{\pmb{e}'}) \rangle_{lg_P} & \text{if $\pmb{e} \not= \pmb{e}'$}.
\end{cases}
\end{align*}
Then, since
\[
\sum_{\pmb{e}, \pmb{e}' \in l\Theta \cap \ZZ^n} a_{\pmb{e}, \pmb{e}'} x_{\pmb{e}} x_{\pmb{e}'} \leq
\sum_{\pmb{e}, \pmb{e}' \in l\Theta \cap \ZZ^n} a'_{\pmb{e}, \pmb{e}'} x_{\pmb{e}} x_{\pmb{e}'},
\]
we have
\begin{align*}
\# \aH_{L^2}(l\overline{P}) & \geq  \# \left\{ (x_{\pmb{e}}) \in \ZZ^{l\Theta \cap \ZZ^n} \ \left|\  \sum\nolimits_{\pmb{e},\pmb{e}' \in l \Theta \cap \ZZ^n} a_{\pmb{e},\pmb{e}'}x_{\pmb{e}}x_{\pmb{e}'} \leq 1 \right\}\right. \\
& \geq
 \# \left\{ (x_{\pmb{e}}) \in \ZZ^{l\Theta \cap \ZZ^n} \ \left|\  \sum\nolimits_{\pmb{e},\pmb{e}' \in l \Theta \cap \ZZ^n} a'_{\pmb{e},\pmb{e}'}x_{\pmb{e}}x_{\pmb{e}'} \leq 1 \right\}\right. .
\end{align*}
On the other hand, by Lemma~\ref{lem:integral:formula:lattice:points},
\begin{multline*}
\liminf_{l\to\infty} \frac{\log  \# \left\{ (x_{\pmb{e}}) \in \ZZ^{l\Theta \cap \ZZ^n} \ \left|\  \sum_{\pmb{e},\pmb{e}' \in l \Theta \cap \ZZ^n} a'_{\pmb{e},\pmb{e}'}x_{\pmb{e}}x_{\pmb{e}'} \leq 1 \right\}\right.}{l^{n+1}/(n+1)!} \\
\geq \frac{(n+1)!}{2} \int_\Theta  \phi_{(\pmb{x}_1, \varphi_{\pmb{a}'}(\widetilde{\pmb{x}}_1)), \ldots, (\pmb{x}_r, \varphi_{\pmb{a}'}(\widetilde{\pmb{x}}_r))}(\pmb{x}) d\pmb{x},
\end{multline*}
and hence $\avol(\overline{P}) > \avol(\overline{D}_{\pmb{a}}) - \epsilon$ by Lemma~\ref{lem:vol:L:2} and (b).

\end{document}